\documentclass[10pt, letterpaper,reqno]{amsart}
\usepackage[left=1in,right=1in,bottom=0.5in,top=0.8in]{geometry}
\usepackage{amsfonts}
\usepackage{amsmath, amssymb,mathtools}
\usepackage{graphicx}
\usepackage[font=small,labelfont=bf]{caption}
\usepackage{epstopdf}
\usepackage{xcolor}
\usepackage{amsthm}
\usepackage{float}
\usepackage{enumitem}
\usepackage{pgfplots}
\usepackage{listings}
\usepackage{longtable}
\usepackage{mathrsfs}
\usepackage{dsfont}
\usepackage{mathtools}
\usepackage{dsfont} 
\usepackage{tipa}
\usepackage{xfrac} 
\usepackage[normalem]{ulem}
\usepackage[pdfpagelabels,hyperindex]{hyperref}
\hypersetup{linkbordercolor=green}

\newcommand*{\rom}[1]{\expandafter\@slowromancap\romannumeral #1@}
\DeclareMathAlphabet{\mathpzc}{OT1}{pzc}{m}{it}
 \colorlet{lgray}{white!80!black}
\colorlet{lred}{white!85!red}
\colorlet{lgreen}{white!60!green}
\colorlet{dgreen}{black!30!green}
\colorlet{lpurple}{white!60!purple}
\colorlet{lblue}{white!60!blue}
\definecolor{green}{rgb}{0.1,0.8,0.1}
\definecolor{yellow}{rgb}{1.0,0.85,0.25}
\definecolor{purple}{rgb}{1.0, 0, 1.0}
\definecolor{blue}{rgb}{0, 0, 1.0} 
\tikzstyle{unfused}=[lgray, line width=1.5pt, ->]
\tikzstyle{fused}=[lgray, line width=4pt, ->]
\tikzstyle{dual}=[black, line width=1pt, dashed]
\tikzstyle{lightdual}=[black, line width=0.5pt, dashed]
\tikzstyle{cut}=[black, line width=1.0pt]

\theoremstyle{plain}
\newtheorem{theorem}{Theorem}[section]
\newtheorem{lemma}[theorem]{Lemma}
\newtheorem{proposition}[theorem]{Proposition}
\newtheorem{corollary}[theorem]{Corollary}

\theoremstyle{definition}
\newtheorem{definition}[theorem]{Definition}
\newtheorem{remark}[theorem]{Remark}

\numberwithin{equation}{section}
\pgfplotsset{compat=1.9} 
\DeclareFontFamily{U}{mathx}{}
\DeclareFontShape{U}{mathx}{m}{n}{<-> mathx10}{}
\DeclareSymbolFont{mathx}{U}{mathx}{m}{n}
\DeclareMathAccent{\widehat}{0}{mathx}{"70}
\DeclareMathAccent{\widecheck}{0}{mathx}{"71}  
\def \tw {w}
\def \bw {\overline{w}}

\def \bc {\overline{c}}
\def \C {\mathcal{C}}
\def \lb {\left(}
\def \rb {\right)}
\def \be {\begin{equation}}
\def \ee {\end{equation}}
\def \k {\kappa}
\def \de {\delta}
\def \e {\mathfrak{e}}
\def \f {\mathfrak{f}}
\def \E {\mathtt{E}}
\def \WO {\widetilde{\Omega}}
\def \ep {\varepsilon}
\def \aa {\mathbf{a}}
\def \bb {\mathbf{b}}
\def \cc {\mathbf{c}}
\def \dd {\mathbf{d}}
\def \eee{\mathbf{e}}
\def \fff{\mathbf{f}} 
\def \RR {\mathbb{R}}
\def \NN {\mathbb{N}_0}
\def \i {\mathtt{i}}
\def \fc{\mathfrak{c}}
\def \tc {\widetilde{c}}
\def \td {\widetilde{d}}
\def \bpi {\pi} 
\def \bP {P} 
\def \ll{\langle}
\def \rr{\rangle}
\def \p{\mathbf{p}}
\def \x {\mathbf{x}}
\def \y {\mathbf{y}} 
\def \ep {\varepsilon}
\def \D {\mathbf{D}}
\def \E {\mathbf{E}} 
\def \tE {\mathtt{E}}
\def \xx {\mathtt{x}}

\def \d {\mathsf{d}}  
\def \B {\mathbb{B}}
\def \one {\mathds{1}}
\def \c {\boldsymbol{c}}
\def \xx {\boldsymbol{x}}
\def \ph {\Phi}
\def \le {\left[}
\def \re {\right]}
\def \pp {\mathfrak{p}}
\def \bI {\mathbf{I}}
\def \co {\eta^A}
\def \pone {\partial_{t=1}}
\def \al {\theta}
\def \nal {n^{\theta}}

\def\floor#1{\left\lfloor #1 \right\rfloor}
\DeclareMathOperator{\dehp}{DEHP}
\DeclareMathOperator{\usw}{USW}
\DeclareMathOperator{\res}{Res}
\DeclareMathOperator{\LL}{L}
\DeclareMathOperator{\HH}{H}
\DeclareMathOperator{\CL}{CL}
\DeclareMathOperator{\cont}{cont} 
\newcommand{\upperRomannumeral}[1]{\uppercase\expandafter{\romannumeral#1}}
\usepgflibrary{fpu}
\makeatletter
\let\NAT@parse\undefined
\makeatother

\title{Askey--Wilson signed measures and open ASEP in the shock region}
\author{Yizao Wang}
\address
{
Yizao Wang\\
Department of Mathematical Sciences\\
University of Cincinnati\\
2815 Commons Way\\
Cincinnati, OH, 45221-0025, USA.
}
\email{yizao.wang@uc.edu}
\author{Jacek Weso\l owski}
\address{ 
Jacek Weso\l owski\\
 Faculty of Mathematics and Information Science\\
Warsaw University of Technology and Statistics\\
Poland,
Warszawa, Poland} 
\email{jacek.wesolowski@pw.edu.pl} 
\author{Zongrui Yang}
\address
{
Zongrui Yang\\
Department of Mathematics\\
Columbia University\\
2990 Broadway\\
New York, NY, 10027-7110, USA.
}
\email{zy2417@columbia.edu}
 
\begin{document}
\begin{abstract}
We introduce a family of  multi-dimensional Askey--Wilson signed measures. We offer an  explicit description of the stationary measure of the open asymmetric simple exclusion process (ASEP) in the full phase diagram, in terms of integrations with respect to these Askey--Wilson signed measures. Using our description, we   provide a rigorous derivation of the density profile and limit fluctuations of open ASEP  in the entire shock region, including the high and low density phases as well as the coexistence line. This in particular confirms the existing physics postulations of the density profile.
\end{abstract}
\maketitle
\section{Introduction and main results}
\subsection{Preface}
The open asymmetric simple exclusion process (ASEP) is a paradigmatic model for non-equilibrium systems with open boundaries and, asymptotically, for KPZ universality. In the last 30 years, extensive studies have been devoted to understanding its stationary measure through the `matrix product ansatz' approach introduced in the seminal work \cite{derrida1993exact} by  B. Derrida,
 M. Evans, V. Hakim and V. Pasquier. Based on the relations with 
 Askey--Wilson polynomials found in \cite{uchiyama2004asymmetric} by T. Sasamoto, M. Uchiyama and M.~Wadati, 
\cite{bryc2017asymmetric} expressed the joint generating function of the stationary measure in terms of expectations of the Askey--Wilson processes introduced in \cite{bryc2010askey}.
The   phase diagram of open ASEP consists of the fan region and the shock region, and 
the explicit characterization by \cite{bryc2017asymmetric}
 is only available in the fan region because of certain constraints on parameters which guarantee the positivity of the Askey--Wilson measures.
This is a powerful method that allows the rigorous derivations of  many   asymptotics (in the macroscopic scale) of the open ASEP stationary measure in the fan region, including the large deviation \cite{bryc2017asymmetric} and density profile and limit fluctuations \cite{bryc2019limit}. By heavily exploiting  this method, the stationary measure of open KPZ equation was first described by \cite{corwin2021stationary} under boundary parameter range $u+v\geq 0$ (which come from the fan region of open ASEP under the weakly asymmetric scaling \cite{corwin2018open}), see  also related works \cite{bryc2021markov,bryc2021markov2,barraquand2021steady} and review \cite{corwin2022some}. 

In the shock region, however, the macroscopic asymptotic behavior of stationary measure is less rigorously understood. The shock region consists of portions of the high and low density phases, as well as the coexistence line. On the high and low density phases, it is widely accepted in the physics literature that the density profile is a constant, depending on the bulk and boundary parameters (for a non-exhaustive list, see \cite{derrida1993exact,schutz1993phase,sasamoto2000density,derrida2003exact,uchiyama2004asymmetric} and references in  surveys \cite{blythe2007nonequilibrium,corwin2022some}).
We are unable to find previous results on the second order fluctuation limits of open ASEP stationary measure in the shock region. On the coexistence line, the density profile is a random process, depending on the random position of the shock that is uniformly distributed along the system. 
 This density profile was predicted in physics works \cite{derrida1993exact,schutz1993phase,essler1996representations,mallick1997finite,derrida2002exact,derrida2003exact} utilizing the matrix product ansatz and  physical heuristics, assuming different special parameter conditions.



In this paper we offer an explicit description of the open ASEP stationary  
measures
 in  full phase diagram, by modifying and extending the techniques in \cite{bryc2010askey,bryc2017asymmetric}.
This description puts both the fan region and shock region in a unified framework, which enables  rigorous derivations of the asymptotics.
Instead of expectations of Askey--Wilson processes, we describe the joint generating  
functions
 of the stationary measure in terms of integrations with respect to certain multi-dimensional Askey--Wilson signed measures. 
To demonstrate the usefulness of this method, we investigate the density profiles and fluctuations in the shock region. 
Our results confirm the constant and random density profiles postulated in the aforementioned physics literature, which apply respectively to the high and low-density phases, as well as on the coexistence line. We also obtain the limit fluctuations on the high and low density phases, which are the same as in the fan region \cite{bryc2019limit}, given by Brownian motions. To our knowledge, our work here is the first rigorous derivation of macroscopic density profiles and fluctuation limits for open ASEP in the entire shock region.


We expect that the explicit description given in this paper should have other applications. Many other most commonly investigated statistics could be studied by this method, including currents, multi-point correlation functions and large deviations. By taking different scaling limits, one could possibly compute the multi-point Laplace transform of open KPZ equation stationary measure and the (conjectural) open KPZ fixed point stationary measure, extending the formulas respectively in \cite{corwin2021stationary} and \cite{bryc2022asymmetric} for $u+v\geq0$ to the full phase diagram. Then it remains the work of inverting those multi-point Laplace transforms as in \cite{bryc2021markov,barraquand2021steady,bryc2021markov2} and proving the conjectural descriptions in \cite{barraquand2021steady}.



\subsection{Backgrounds} 
The open asymmetric simple exclusion process (ASEP) is a continuous-time irreducible Markov process on state space $\{0,1\}^n$ with parameters
\be\label{eq:conditions open ASEP}
\alpha,\beta>0,\quad \gamma,\de\geq 0,\quad 0\leq q<1,
\ee  
which models the evolution of particles on the lattice $\left\{1,\dots,n\right\}$. Particles are allowed to move to its nearest left/right neighbors and can also enter or exit the system at two boundary sites $1, n$. Specifically, particles move at random to the left with rate $q$ and to the right with rate $1$, but a move is prohibited (excluded) if the target site is already occupied. Particles enter at random the system and are placed at site $1$ with rate $\alpha$ and at site $n$ with rate $\delta$, provided that the site is empty. Particles are also removed at random with rate $\gamma$ from site $1$ and with rate $\beta$ from site $n$. These jump rates are summarized in Figure \ref{fig:openASEP}.  
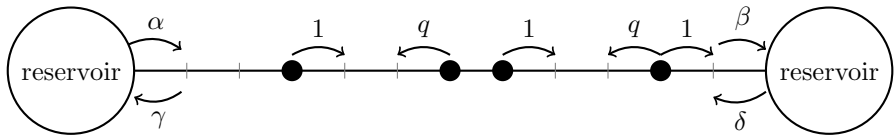
\begin{figure}[h]
\centering
\begin{tikzpicture}[scale=0.7]
\draw[thick] (-1.2, 0) circle(1.2);
\draw (-1.2,0) node{reservoir};
\draw[thick] (0, 0) -- (12, 0);
\foreach \x in {1, ..., 12} {
	\draw[gray] (\x, 0.15) -- (\x, -0.15);
}
\draw[thick] (13.2,0) circle(1.2);
\draw(13.2,0) node{reservoir};
\fill[thick] (3, 0) circle(0.2);
\fill[thick] (6, 0) circle(0.2);
\fill[thick] (7, 0) circle(0.2);
\fill[thick] (10, 0) circle(0.2);
\draw[thick, ->] (3, 0.3)  to[bend left] node[midway, above]{$1$} (4, 0.3);
\draw[thick, ->] (6, 0.3)  to[bend right] node[midway, above]{$q$} (5, 0.3);
\draw[thick, ->] (7, 0.3) to[bend left] node[midway, above]{$1$} (8, 0.3);
\draw[thick, ->] (10, 0.3) to[bend left] node[midway, above]{$1$} (11, 0.3);
\draw[thick, ->] (10, 0.3) to[bend right] node[midway, above]{$q$} (9, 0.3);
\draw[thick, ->] (-0.1, 0.5) to[bend left] node[midway, above]{$\alpha$} (0.9, 0.4);
\draw[thick, <-] (0, -0.5) to[bend right] node[midway, below]{$\gamma$} (0.9, -0.4);
\draw[thick, ->] (12, -0.4) to[bend left] node[midway, below]{$\de$} (11, -0.5);
\draw[thick, <-] (12, 0.4) to[bend right] node[midway, above]{$\beta$} (11.1, 0.5);
\end{tikzpicture}
\caption{Jump rates in the open ASEP.}
\label{fig:openASEP}
\end{figure} 

Throughout the paper we assume condition \eqref{eq:conditions open ASEP} and work with the following parameterization: We set
 $$
\k_{\pm}(x,y)=\frac{1}{2x}\lb1-q-x+y\pm\sqrt{(1-q-x+y)^2+4xy}\rb, \quad \mbox{for }\; x>0\mbox{ and }y\geq 0,
$$ 
and denote
\be\label{eq:defining ABCD}
A=\k_+(\beta,\de),\quad B=\k_-(\beta,\de),\quad C=\k_+(\alpha,\gamma),\quad D=\k_-(\alpha,\gamma).
\ee
The quantities $\frac{A}{1+A}$ and $\frac{1}{1+C}$ defined by the parameters above have nice
physical interpretations as the `effective densities' near the left and right boundaries of the system, see for example the review paper \cite[Section 6.2]{corwin2022some}.

One can check that \eqref{eq:defining ABCD} gives a bijection between \eqref{eq:conditions open ASEP} and 
\be\label{eq:conditions qABCD}
A,C\geq0,\quad -1<B,D\leq 0,\quad 0\leq q<1.
\ee 
We will assume \eqref{eq:conditions open ASEP} and consequently,  \eqref{eq:conditions qABCD} throughout the paper. 

It has been known since \cite{derrida1993exact} that the phase diagram of open ASEP involves only two boundary parameters $A,C$, and exhibits three phases (see Figure \ref{fig:phase diagram density fluctuation}):
\begin{enumerate}
        \item [$\bullet$] (maximal current phase) $A<1$, $C<1$,
        \item [$\bullet$] (high density phase) $A>1$, $A>C$,
        \item [$\bullet$] (low density phase) $C>1$, $C>A$.
    \end{enumerate}
The boundary $A=C>1$ between the high and low density phases is called the coexistence line.
    
There are also two regions on the phase diagram distinguished by \cite{derrida2002exact,derrida2003exact}: 
    \begin{enumerate}
        \item [$\bullet$] (fan region) $AC<1$,
        \item [$\bullet$] (shock region) $AC>1$.
    \end{enumerate}

We denote by $\mu_n$ the (unique) stationary measure of open ASEP, which is a probability measure on $(\tau_1,\dots,\tau_n)\in\left\{0,1\right\}^n$, where
$\tau_i\in\left\{0,1\right\}$ is the occupation variable on site $i$, for $1\leq i\leq n$.
A `matrix product ansatz' characterization of the stationary measure $\mu_n$ was given in the seminal work \cite{derrida1993exact} by B. Derrida, M. Evans, V. Hakim and V. Pasquier:  
Assume  there are matrices $\D$, $\E$, a row vector $\ll W|$ and a column vector $|V\rr$ with the same (possibly infinite) dimension, satisfying the following:
\be\label{eq:DEHP algebra} 
        \D\E-q\E\D=\D+\E, \quad
        \ll W|(\alpha\E-\gamma\D)=\ll W|, \quad
        (\beta\D-\de\E)|V\rr=|V\rr 
 \ee
    (which is commonly referred to as the $\dehp$ algebra). Then
for any $t_1,\dots,t_n>0$, we have  
 \be \label{eq:MPA open ASEP}
    \mathbb{E}_{\mu_n}\le\prod_{i=1}^nt_i^{\tau_i}\re
    =\frac{\ll W|(\E+t_1\D)\times\dots\times(\E+t_n\D)|V\rr}{\ll W|(\E+\D)^n|V\rr},
    \ee 
assuming that the denominator $\ll W|(\E+\D)^n|V\rr$ is nonzero. 
 
 To make use of the matrix product ansatz \eqref{eq:MPA open ASEP} for the stationary measure of an open ASEP, one needs to 
 find concrete representations of \eqref{eq:DEHP algebra}.  
Over time various  
representations
have been discovered for different parameters $(q,\alpha,\beta,\gamma,\delta)$.
See for example \cite{derrida1993exact,uchiyama2004asymmetric,enaud2004large,sasamoto1999one,sandow1994partially,blythe2000exact,essler1996representations} for infinite dimensional representations and \cite{mallick1997finite,essler1996representations} for finite dimensional ones.
Some of these representations enable calculations of several asymptotics for open ASEP.
One such representation for general parameters $(q,\alpha,\beta,\gamma,\delta)$  
obtained in the seminal (physics) work \cite{uchiyama2004asymmetric} by M. Uchiyama, T. Sasamoto and M. Wadati, was related to the Askey--Wilson orthogonal polynomials (with parameters $(A,B,C,D,q)$ depending on the open ASEP parameters via \eqref{eq:defining ABCD}; see Section \ref{subsec:Askey--Wilson polynomials} and Section \ref{subsubsec:USW representation} for a brief review). It  is often referred to as the USW representation. Using this representation, the  $n\rightarrow\infty$ asymptotics of the mean and variance of particle density and mean current was studied therein.

 In a different line of research, \cite{bryc2010askey}   observed that the Askey--Wilson polynomials with parameters $(A,B,C,D,q)$, under suitable conditions, are  orthogonal with respect to a unique compactly supported probability measure on $\RR$, referred to as the Askey--Wilson measure. Then, \cite{bryc2010askey} introduced the Askey--Wilson process with parameters $(A,B,C,D,q)$ as a time-inhomogeneous Markov process with specific one-dimensional marginal laws and transition probabilities, denoted by $\pi_t(\d y)$ and $P_{s,t}(x,\d y)$, given by the Askey--Wilson measures depending on $A,B,C,D,q$ and $s,t,x$ (for explicit expressions, see \eqref{eq:marginal AW signed} and \eqref{eq:transition AW signed}). All these are independent from the investigations of open ASEP, but instead following earlier developments on so-called quadratic harnesses, see for example \cite{bryc2007quadratic,bryc2008bi,bryc2005conditional,bryc2015infinitesimal}. 

 Next, when the Askey--Wilson process was related to open ASEP via \eqref{eq:defining ABCD} in \cite{bryc2017asymmetric}, it turned out that the restrictions on parameters that guarantee the existence of Askey--Wilson process become $AC<1$, corresponding exactly to the fan region of the open ASEP. 
By combining several properties of Askey--Wilson polynomials and processes, under the $\usw$ representation of $\D,\E,\ll W|$ and $|V\rr$, the joint  generating function of open ASEP stationary measure (given in \eqref{eq:MPA open ASEP}) was characterized as: For $0<t_1\leq\dots\leq t_n$,
\be\label{eq:MPA as AW process}
(1-q)^n\ll W|(\E+t_1\D)\times\dots\times(\E+t_n\D)|V\rr=\mathbb{E}\lb\prod_{i=1}^n(1+t_i+2\sqrt{t_i}\,Y_{t_i})\rb
\ee
where $\lb Y_t\rb_{t\geq 0}$ is the 
Askey--Wilson
 process with parameters $(A, B, C, D, q)$.

The identity \eqref{eq:MPA as AW process} is a powerful representation which made it possible to have rigorous analysis of many asymptotics of open ASEP stationary measure in the fan region, including large deviation in \cite{bryc2017asymmetric} and density profile and limit fluctuations in \cite{bryc2019limit}.
By (heavily) exploiting this method, the stationary measure of many other models in the KPZ class can also be studied. The stationary measure of open KPZ equation on $[0,1]$ with general Neumann boundary conditions was first constructed by \cite{corwin2021stationary} under boundary parameter range $u+v\geq 0$ (which come from the fan region of open ASEP under the weakly asymmetric scaling \cite{corwin2018open}), see  also related works \cite{bryc2021markov,bryc2021markov2,barraquand2021steady} and review \cite{corwin2022some}. The scaling convergence of open ASEP stationary measure to the (conjectural) stationary
measure of the open KPZ fixed point (postulated in \cite{barraquand2021steady}) was proved in \cite{bryc2022asymmetric} under $u+v\geq 0$. The stationary measure of the `six-vertex model' on a strip with two open boundaries was studied in \cite{yang2022stationary}.

\subsection{Main results}
 Note that the right-hand side of \eqref{eq:MPA as AW process} can be written as an integral. Namely, for $\pi_t(\d y)$ and $P_{s,t}(x,\d y)$  as the one-dimensional marginal laws and transitional probabilities of the Askey--Wilson process $Y$, they determine the multi-dimensional marginal laws of the Askey--Wilson process: 
\[
\pi_{t_1,\dots,t_n}(\d x_1,\dots,\d x_n):=\pi_{t_1}(\d x_1)P_{t_1,t_2}(x_1,\d x_2)\dots P_{t_{n-1},t_n}(x_{n -1},\d x_n).
\] 
In this way, the right-hand side of \eqref{eq:MPA as AW process} can be written as an integration of the measure $\pi_{t_1,\dots,t_n}$. 

Our first contribution is to show that in the shock region, one can still find $\pi_t(\d y)$ and $P_{s,t}(x,\d y)$ as Askey--Wilson {\em signed} measures, define $\pi_{t_1,\dots,t_n}$ accordingly, and establish a counterpart of \eqref{eq:MPA as AW process}. 
 Namely, it turned out that the formulae for $\pi_t$ and $P_{s,t}$ in \cite{bryc2017asymmetric} can be extended to the region $AC>1$ to correspond to certain signed measures (see \eqref{eq:marginal AW signed} and \eqref{eq:transition AW signed} below), and they are exactly the measures to work with in the shock region. 
 All our analysis later shall be based on $\pi_{t_1,\dots,t_n}$, which now becomes a natural extension to the earlier studies. The analysis becomes also more involved because the measures are no longer probability ones. We provide the basic tools related to Askey--Wilson signed measures in Section \ref{sec:AW polynomials and signed measures}. We actually see more similarities between the fan and the shock regions. 
For example, $\{P_{s,t}\}_{s<t}$ 
satisfy the Chapman--Kolmogorov equation {\em formally} (these measures are with  total mass one but they have strictly negative parts when $AC>1$). Obviously, an interpretation of the corresponding Askey--Wilson Markov process is lost. 

 Working with the extended formulae for $\pi_t$ and $P_{s,t}$, the first main theorem provides an explicit and unified description 
of the open ASEP stationary measure in the full parameter space in terms of multiple integrals with respect to certain multi-dimensional Askey--Wilson signed measures. We  only need one technical constraint $ABCD\notin\{q^{-l}:l\in\NN\}$.  In this paper we use the convention that $\NN=\{0,1,\ldots\}$.

\begin{theorem}\label{thm: stationary measure in terms of Askey--Wilson integral}
Assume \eqref{eq:conditions qABCD} and $ABCD\notin\{q^{-l}:l\in\NN\}$.   
Then there exists a polynomial $\Pi_n$ of $n$ variables $t_1,\dots,t_n$ with coefficients depending on $A,B,C,D,q$, such that for any $t_1,\dots,t_n>0$, 
\be \label{eq:MPA open ASEP1}
    \mathbb{E}_{\mu_n}\le\prod_{i=1}^nt_i^{\tau_i}\re
    =\frac{\Pi_n\lb t_1,\dots,t_n\rb}{\Pi_n(1,\dots,1)}.
\ee
There exists an open interval $I$ containing $(1-\ep,1)$ for some $\ep>0$, such that for any $t_1\leq\dots\leq t_n$ in $I$,
\be\label{PIN}
    \Pi_n\lb t_1,\dots,t_n\rb=\int_{\RR^n}\prod_{i=1}^n(1+t_i+2\sqrt{t_i}x_i)\pi_{t_1,\dots,t_n}(\d x_1,\dots,\d x_n),\ee
where $\pi_{t_1,\dots,t_n}$ is a finite signed measure  (depending on $q,A,B,C,D$ and  defined in Section \ref{subsec:The multi-time Askey--Wilson signed measures}), which is compactly supported in $\RR^n$ and has a total mass $1$
(i.e. $\int_{\RR^n}\pi_{t_1,\dots,t_n}(\d x_1,\dots,\d x_n)=1$).

Furthermore, when we (additionally) assume 
$A/C\notin\{q^l:l\in\mathbb{Z}\}$ if $A,C\geq 1$,
then   the open interval $I$ above can be chosen such that $1\in I$. \end{theorem}

    
 
\begin{remark}
    As we will see in the proof (in Section \ref{sec:Proof of Theorem stationary measure in terms of Askey--Wilson integral}), the polynomial $\Pi_n\lb t_1,\dots,t_n\rb$ is actually given by the LHS of \eqref{eq:MPA as AW process}:
    $$
    \Pi_n\lb t_1,\dots,t_n\rb=(1-q)^n\ll W|(\E+t_1\D)\times\dots\times(\E+t_n\D)|V\rr,
    $$
for $\D$, $\E$, $\ll W|$ and $|V\rr$ satisfying the $\dehp$ algebra. By \cite[Appendix A]{mallick1997finite} this quantity does not depend on the specific representation of the $\dehp$ algebra and in particular we will use the USW representation \cite{uchiyama2004asymmetric} (see Section \ref{subsubsec:USW representation}) in the proof. 
It has been noted in \cite{mallick1997finite,essler1996representations} that the matrix product ansatz \eqref{eq:MPA open ASEP1} does not work (i.e. the denominator $\Pi_n(1,\dots,1)$ may equal to zero) when $ABCD=q^{-l}$ for some $l\in\NN$. Such cases are referred to as the `singular' cases of the matrix ansatz, for which an alternative
method is developed in \cite{bryc2019matrix}. In this paper we only deal with the `non-singular' case, thus assuming $ABCD\notin\{q^{-l}:l\in\NN\}$ in Theorem \ref{thm: stationary measure in terms of Askey--Wilson integral}.  
\end{remark}
\begin{remark}
We reiterate that when $AC<1$, the expression \eqref{PIN} has a probabilistic interpretation (see \eqref{eq:MPA as AW process}) and has been established in \cite{bryc2017asymmetric}. The probabilistic interpretation of the integral is lost when $AC>1$. When $AC=1$, the stationary measure $\mu_n$ becomes a product of i.i.d. Bernoulli random variables $\tau_1,\dots,\tau_n$ with mean $A/(1+A)=1/(1+C)$, see for example \cite[Appendix A]{enaud2004large}.
\end{remark} 

As an application, we fix all other parameters $q,\alpha,\beta,\gamma,\delta$ and take $n\rightarrow\infty$ in the open ASEP stationary measure. In the next two theorems, we study the first and second order asymptotics on the high and low density phases and the coexistence line.  

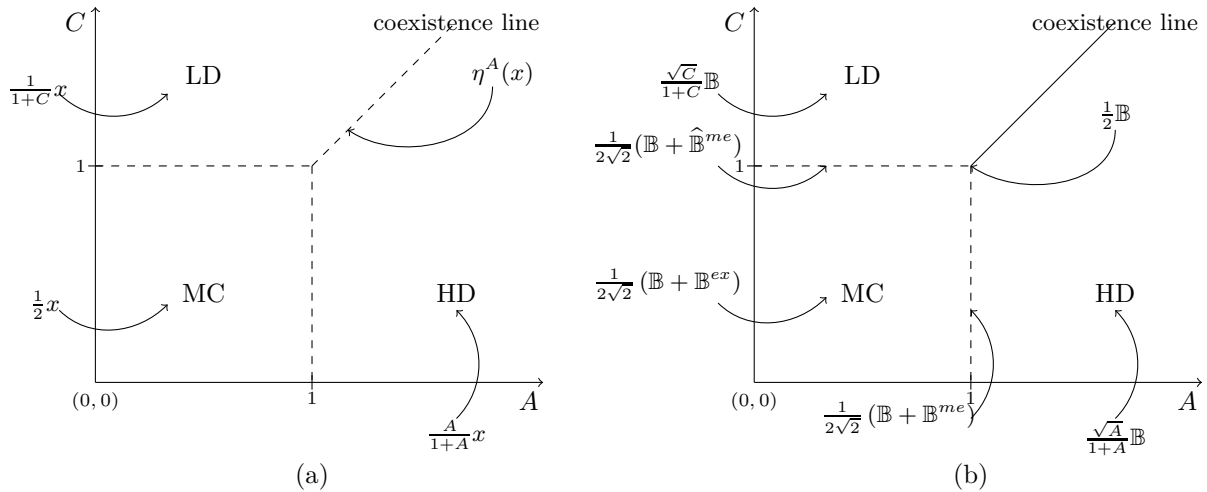
\begin{figure} 
\usetikzlibrary{patterns}
    \centering
    \begin{tikzpicture}[scale=0.96] 
 \draw[->] (5,5) to (5,10.2);
 \draw[->] (5.,5) to (11.2,5);
   \draw[-, dashed] (5,8) to (8,8);
   \draw[-, dashed] (8,8) to (8,5);
   \draw[-, dashed] (8,8) to (10,10);
   \node [left] at (5,8) {\scriptsize$1$};
   \node[below] at (8,5) {\scriptsize $1$};
     \node [below] at (11,5) {$A $};
   \node [left] at (5,10) {$C $};  
  \node [above] at (4.2,8.7) {\small$\frac{1}{1+C}x$};  
  \draw[-] (8,4.9) to (8,5.1);
   \draw[-] (4.9,8) to (5.1,8); 
 \node [below] at (5,5) {\scriptsize$(0,0)$}; ;
     \draw[->] (4.5,9) to [out=-45,in=-135] (6,9);
      \node [below] at (10,4.6) {\small$\frac{A}{1+A}x$}; 
       \draw[->] (10,4.5) to [out =45,in= -45] (10,6) ;
           \node [above] at (6.5,9) {LD}; 
    \node [below] at (10,6.5) {HD};  
     \node [below] at (6.5,6.5) {MC};  
     \draw[->] (4.5,6) to [out=-45,in=-135] (6,6.08);
  \node [above] at (4.3,5.75) {\small$\frac{1}{2}x$};
  \draw[->] (10.5,9.1) to [out=-90,in=-35] (8.5,8.5);
  \node at (10.64,9.3) {\small $\eta^A(x)$};
  \node at (10,10) {\small coexistence line };
  \node at (8,3.7) {(a)};
\end{tikzpicture}
\quad
\begin{tikzpicture}[scale=0.96] 
 \draw[->] (5,5) to (5,10.2);
 \draw[->] (5.,5) to (11.2,5);
   \draw[-, dashed] (5,8) to (8,8);
   \draw[-, dashed] (8,8) to (8,5); 
   \draw[-] (8,8) to (10,10);
   \node [left] at (5,8) {\scriptsize$1$};
   \node[below] at (8,5) {\scriptsize $1$};
     \node [below] at (11,5) {$A$};
   \node [left] at (5,10) {$C$}; 
    \node [above] at (6.5,9) {LD}; 
    \node [below] at (10,6.5) {HD};  
     \node [below] at (6.5,6.5) {MC};   
     \node at (10,10) {\small coexistence line }; 
 \draw[->] (4.5,8) to [out=-45,in=-135] (6,8);
  \node [above] at (3.8,7.9) {\small $\frac{1}{2\sqrt{2}}(\mathbb B+\widehat{\mathbb B}^{me})$}; 
  \node[left] at (8.2,4.5) {\small $\frac{1}{2\sqrt{2}}\lb\mathbb B+{\mathbb B}^{me}\rb$};
 \draw[->] (8,4.5) to [out =45,in= -45] (8,6) ;
  \draw[-] (8,4.9) to (8,5.1);
   \draw[-] (4.9,8) to (5.1,8); 
 \node [below] at (5,5) {\scriptsize$(0,0)$};
     \node [above] at (4.1,8.8){ \small $\frac{\sqrt{C}}{1+C}\mathbb B$};  
     \draw[->] (4.5,9) to [out=-45,in=-135] (6,9); 
     \draw[->] (10,4.5) to [out =45,in= -45] (10,6) ;
     \node [below] at (10,4.6) { \small $\frac{\sqrt{A}}{1+A}\mathbb B$}; 
     \draw[->] (4.5,6.1) to [out=-45,in=-135] (6,6.18);
  \node [above] at (3.8,6) {\small $\frac{1}{2\sqrt{2}}\lb\mathbb B+ {\mathbb B}^{ex}\rb$}; 
  \node at (8,3.7) {(b)};
  \draw[->] (10,8.5) to [out=-90,in=-35] (8,8);
  \node at (10,8.7) {\small $\frac{1}{2}\mathbb B$};  
\end{tikzpicture}
    \caption{Phase diagrams for the open ASEP stationary measures. LD, HD, MC respectively stand for the low density, high density and maximal current phases.
    The asymptotic density profile in different phases are indicated in  (a) and  fluctuation limits are indicated in (b). Part (b) complements Figure 2 of \cite{bryc2019limit} with fluctuations for the shock region. Processes $\mathbb{B}$, ${\mathbb B}^{ex}$, ${\mathbb B}^{me}$ and $\widehat{\mathbb B}^{me}$ respectively stand for the Brownian motion, excursion, meander, and reversed meander, see for example \cite{bryc2019limit} for their definitions. Processes in the sums are assumed to be independent. Note that $A=C>1$ is excluded from (b) since  on the coexistence line, the  asymptotic density profile becomes random, see \eqref{eq:coexistence line process}, and  it is not clear how to make sense of  fluctuations therein.}
    \label{fig:phase diagram density fluctuation}
\end{figure}

\begin{theorem}\label{thm:fluctuations}  
We introduce the centered height functions:
\be  \label{eq:centered height functions}
h_n^{\HH}(x)=\sum_{i=1}^{\floor{nx}}\lb \tau_i-\frac{A}{1+A}\rb,\quad
h_n^{\LL}(x)=\sum_{i=1}^{\floor{nx}}\lb \tau_i-\frac{1}{1+C}\rb.
\ee
In the high density phase,  i.e. $A>C$ and $A>1$, we have, as $n\rightarrow\infty$,
  \be \label{eq:HD fluc}
   \frac{1}{\sqrt{n}} \{ h^{\HH}_n(x)\}_{x\in[0,1]}\stackrel{\mathrm{f.d.d.}}{\Longrightarrow}   \frac{\sqrt{A}}{1+A}\{\B(x)\}_{x\in[0,1]}. 
  \ee
In the low density phase,  i.e. $C>A$ and $C>1$, we have, as $n\rightarrow\infty$,
  \be \label{eq:LD fluc}
   \frac{1}{\sqrt{n}} \{ h^{\LL}_n(x)\}_{x\in[0,1]}\stackrel{\mathrm{f.d.d.}}{\Longrightarrow}   \frac{\sqrt{C}}{1+C}\{\B(x)\}_{x\in[0,1]}. 
  \ee
 \end{theorem}
\begin{remark}
Consider the ordinary height function $h_n(x)=\sum_{i=1}^{\floor{nx}}\tau_i$. For the high density phase, we have
$$\frac{1}{n}h_n(x)-\frac{A}{1+A}x=\frac{A}{1+A}\frac{\floor{nx}-nx}{n}+\frac{1}{n}h_n^{\HH}(x).$$
A similar relation holds for  low density phase. Fluctuations \eqref{eq:HD fluc} and \eqref{eq:LD fluc} also imply the density profile:
\be\label{eq:macroscopic density profile}\frac{1}{n} h_n(x) \stackrel{\mathrm{P}}{\longrightarrow}\begin{cases}
   \frac{A}{1+A}x  & \quad A>C, A>1 \quad\text{ (high density phase)}, \\
   \frac{1}{1+C}x & \quad C>A, C>1 \quad\text{ (low density phase)},
\end{cases}\ee
for all $x\in [0,1]$, where $\stackrel{\mathrm{P}}{\longrightarrow}$ denotes convergence in probability.
This rigorously confirms the postulation of the density profile on the high and low density phases that has been widely accepted in the physics literature. As pointed out in the preface, for a non-exhaustive list of works, see \cite{derrida1993exact,schutz1993phase,sasamoto2000density,derrida2003exact,uchiyama2004asymmetric} and references in the survey papers \cite{blythe2007nonequilibrium,corwin2022some}.
\end{remark}
 
\begin{theorem}\label{thm:density profile coexistence line}
Assume $A=C>1$ and $ABCD\notin\{q^{-l}:l\in\NN\}$. 
We have, as $n\rightarrow\infty$,
$$
\frac{1}{n}\{h_n(x)\}_{x\in[0,1]}\stackrel{\mathrm{f.d.d.}}{\Longrightarrow}\{\co(x)\}_{x\in[0,1]},
$$
where the process 
\be\label{eq:coexistence line process}
\co(x):=\frac{Ax+(1-A)\lb x\wedge U\rb}{1+A},\quad x\in[0,1],
\ee 
and $U\sim U(0,1)$ is uniformly distributed on $[0,1]$.
We use $x\wedge y$ to denote the minimum of $x$ and $y$.
\end{theorem}
\begin{remark} 
  The density profile (the formal derivative of $\co(x)$ with respect to $x$) over $[0,1]$ has the following interpretation: 
   the density is a constant $\frac 1{1+A}$ over $[0,U)$, and a different constant $\frac A{1+A}$ over $[U,1]$, where $U$ is a random variable uniformly distributed on $[0,1]$, representing the location of the shock. This can be read by re-writing
  \[
  \co(x) = \frac1{1+A}(x\wedge U) + \frac A{1+A}(x- x\wedge U).
  \]
This rigorously confirms the predictions of the density profile on the coexistence line in the physics literature, scattered in several works. Using the matrix ansatz, this density profile has been predicted in \cite{derrida1993exact,schutz1993phase} in the open totally asymmetric simple exclusion process (TASEP) case (where $q=0$ and $\gamma=\delta=0$)  case and later in \cite{essler1996representations,mallick1997finite} under other special parameter conditions (which guarantee the existence of finite dimensional representations of the DEHP algebra). 
Physical explanations also appear in \cite{derrida2002exact,derrida2003exact}, where the large deviation around the density profile is studied.
\end{remark}
  

\begin{remark}
    On the coexistence line the density profile is no longer deterministic (instead, it becomes a random process). Consequently, the definition of limit fluctuations becomes problematic and we do not pursue this issue here further.  
\end{remark}  

\begin{remark}
In the fan region and its boundary $AC\leq 1$, the density profile and limit fluctuations have been rigorously studied in \cite{bryc2019limit} using the aforementioned Askey--Wilson process method. Combining the results therein and the above two theorems, we are able to draw a picture (Figure \ref{fig:phase diagram density fluctuation}) of the first and second order asymptotics of open ASEP (the second order asymptotics on coexistence line is out of our reach).
\end{remark}
\begin{remark}
    On a different (microscopic) scale, the density profile has been shown in the mathematical works \cite{liggett1999stochastic,liggett1975ergodic} in the  
    open TASEP case (where $q=0$ and $\gamma=\delta=0$)
    in the full phase diagram. See Theorem 3.29 and Theorem 3.41 in Part \upperRomannumeral{3} of \cite{liggett1999stochastic}. 
\end{remark}
\begin{remark}
Under a special condition $AC=q^{-l}$ for some $l\in\NN$, \cite{mallick1997finite} found finite dimensional representations of the $\dehp$ algebra and studied densities and correlation functions on the shock region. Under the same condition, a recent work   \cite{schuetz2022reverse} found a simple representation of the open ASEP stationary measure on the shock region as a convex combination of Bernoulli shock measures, using a reverse duality introduced therein. After the first version of the present paper was posted, \cite{nestoridi2023approximating} established several limit theorems of the open ASEP stationary measure,  where the limits are taken in a different   scale. Their results cover many parts of the phase diagram, including a sub-region of the shock region, where the characterization from \cite{schuetz2022reverse} was used.   It would be interesting to see if there are connections of these works to our methods. 
\end{remark}

\subsection{Outline of the paper}
In Section \ref{sec:AW polynomials and signed measures} we introduce the Askey--Wilson polynomials, signed measures and multi-dimensional signed measures and prove several properties. 
In Section \ref{sec:Proof of Theorem stationary measure in terms of Askey--Wilson integral} we prove Theorem \ref{thm: stationary measure in terms of Askey--Wilson integral} by extending the projection formula. 
In Section \ref{sec:limit fluctuation} we prove Theorem \ref{thm:fluctuations} and in Section \ref{sec:coexistence line} we prove Theorem \ref{thm:density profile coexistence line}, both by their multi-point Laplace transforms, and using the total variation bounds from Appendix \ref{sec:total variation bounds}.

\subsection*{Acknowledgements}  
We thank Wlodek Bryc for many helpful conversations and for pointing out an error in a previous version of the paper.
We thank   Ivan Corwin and Dominik Schmid for helpful discussions related to the results in this paper. Y.W.~was partially supported by Army Research Office, US (W911NF-20-1-0139).   
J.W.~was partially supported by IDUB grant no. 1820/366/201/2021, Warsaw Univ. Techn., Poland and by Taft Research Center at University of Cincinnati.
Z.Y.~was partially supported by Ivan Corwin’s NSF grant DMS-1811143 as well as the Fernholz Foundation’s `Summer Minerva Fellows' program.  

\section{Askey--Wilson polynomials and signed measures}
\label{sec:AW polynomials and signed measures}
\subsection{Askey--Wilson polynomials}
\label{subsec:Askey--Wilson polynomials}
The Askey--Wilson polynomials $w_m(x):=w_m(x;a,b,c,d|q)$, $m\in \NN:=\{0,1,\ldots\}$, constitute  a family of orthogonal polynomials introduced by Askey and Wilson \cite{askey1985some}.
When $abcd\notin\{q^{-l}:l\in\NN\}$, they are defined by three-term recurrence:
\be\label{eq:three term recurrence}
A_mw_{m+1}(x)+B_mw_m(x)+C_mw_{m-1}(x)=2xw_m(x), 
\ee
for $m\in\NN$,
with $w_0(x)=1$ and $w_{-1}(x)=0$,
where
\begin{align}
\label{eq:Am}
A_m&=
\frac{1-q^{m-1}abcd}{(1-q^{2m-1}abcd)(1-q^{2m}abcd)},
\\ 
B_m&=
\frac{q^{m-1}}{(1-q^{2m-2}abcd)(1-q^{2m}abcd)}
[(1+q^{2m-1}abcd)(qs+abcds')-q^{m-1}(1+q)abcd(s+qs')],
\\ 
\label{eq:Cm}
C_m&=
\frac{(1-q^m)(1-q^{m-1}ab)(1-q^{m-1}ac)(1-q^{m-1}ad)(1-q^{m-1}bc)
(1-q^{m-1}bd)(1-q^{m-1}cd)}
{(1-q^{2m-2}abcd)(1-q^{2m-1}abcd)},
\end{align}
with $s=a+b+c+d$ and $s'=a^{-1}+b^{-1}+c^{-1}+d^{-1}$.

\subsection{Askey--Wilson signed measures}
\label{subsec:Askey--Wilson measure}
We define the Askey--Wilson signed measures. We do not pursue the general parameters but only define them on the parameter region that will be necessary for open ASEP. 
\begin{definition}\label{def:region omega}
We define $\WO$ to be the set of $(a,b,c,d)\in\mathbb{C}^4$ such that   
\begin{enumerate} 
    \item[\textbf{(1)}]   $a^2,b^2,c^2,d^2,ab,ac,ad,bc,bd,cd\notin\{q^{-l}:l\in\NN\}$,
    \item[\textbf{(2)}]   $abcd\notin\{q^{-l}:l\in\NN\}$,
    \item[\textbf{(3)}]   For any two  distinct $\e,\f\in\{a,b,c,d\}$ such that $|\e|,|\f|\geq 1$,  we have  $\e/\f\notin\{q^l:l\in\mathbb{Z}\}$,
    \item[\textbf{(4)}]  $a,b$ are real, and $c,d$ are either real or form complex conjugate pair; $ab<1$ and $cd<1$.
\end{enumerate}
We define $\Omega$ to be the set $(a,b,c,d)$ where only \textbf{(2)}, \textbf{(3)}  and \textbf{(4)} are satisfied.
\end{definition}
 
\begin{definition}\label{def:definition of Askey--Wilson measure}
    Assume $(a,b,c,d)\in\Omega$. We first define a set $F(a,b,c,d)\subset\RR\setminus(-1,1)$ of atoms. The atoms are generated by each $\e\in\{a,b,c,d\}$ with  $|\e|\geq 1$. By condition \textbf{(4)} in Definition \ref{def:region omega}, such $\e$ must be a real number. The atoms generated by $\e$ are given by 
    $$
y_{j}^\eee=y_{j}^\eee(a,b,c,d)=\frac{1}{2}\lb \e q^j+(\e q^j)^{-1}\rb,
$$
with $j\geq0$ such that $|aq^j|\geq1$. 
    The bold symbol $\eee$ in the superscripts signal that the atom is caused by the parameter with label $\eee\in\{\aa,\bb,\cc,\dd\}$. We remark that condition \textbf{(3)} in Definition \ref{def:region omega} above guarantees that the atoms do not collide with each other, i.e. for any two $\eee,\fff\in\{\aa,\bb,\cc,\dd\}$ such that the corresponding $\e,\f$ satisfy $|\e q^j|, |\f q^k|\geq1$ for some $j,k\in\NN$, one has $\e q^j\neq \f q^k$ hence $y_j^{\eee}\neq y_k^{\fff}$.

    The Askey--Wilson signed measure is of mixed type:
    $$\nu(\d x;a,b,c,d)=f(x;a,b,c,d)\one_{|x|<1}\d x+\sum_{x\in F(a,b,c,d)}
p(x) \delta_{x}$$
where the continuous part density is defined as, for $x=\cos\theta\in(-1,1)$,
\be\label{eq:continuous part density}f(x;a,b,c,d)=\frac{(q, ab, ac, ad, bc, bd, cd)_{\infty}}{2\pi
(abcd)_{\infty}\sqrt{1-x^2}} \biggl|\frac{(e^{2\i\theta
})_{\infty}}{(ae^{\i\theta}, be^{\i\theta}, ce^{\i\theta},
de^{\i\theta})_{\infty}} \biggr|^2. 
\ee 
Here and below, for complex $z$ and $n\in\NN\cup\{\infty\}$, we use the $q$-Pochhammer symbol:
$$
(z)_n=(z;q)_n=\prod_{j=0}^{n-1}\,(1-z q^j), \quad (z_1,\cdots,z_k)_n=(z_1,\cdots,z_k;q)_n=\prod_{i=1}^k(z_i;q)_n.
$$
The atoms generated by $\e=a$ (if exist) have corresponding masses 
\begin{align}\label{eq: p_0}
    p(y_{0}^\aa)&=p_0^\aa(a,b,c,d)=\frac{(a^{-2},bc,bd,cd)_\infty}{(b/a,c/a,d/a,abcd)_\infty},\\
    \label{eq: p_j}
p(y_{j}^\aa)&=p_j^\aa(a,b,c,d)=p_0^\aa(a,b,c,d)\frac{q^j(1-a^2q^{2j})(a^2,ab,ac,ad)_j}{(q)_j(1-a^2)a^j\prod_{l=1}^j\lb(b-q^la)(c-q^la)(d-q^la)\rb},\quad j\geq 1. 
\end{align}
We reiterate that the bold symbol $\aa$ in the superscripts signal that the atom (if exists) is caused by the parameter with label $\aa$ (i.e., the first parameter).
 
When $abcd\neq 0$, the formula of $p_j^\aa(a,b,c,d)$ for $j\geq 1$ can be written in a more 
succinct
 form
$$p_j^\aa(a,b,c,d)=p_0^\aa(a,b,c,d)\frac{(a^2,ab,ac,ad)_j(1-a^2
q^{2j})}{(q,qa/b,qa/c,qa/d)_j(1-a^2)} \lb\frac{q}{abcd} \rb^j.$$

For $\e\in\{b,c,d\}$, atoms $y_j^\bb$, $y_j^\cc$, $y_j^\dd$ and masses $p(y_j^\bb)$, $p(y_j^\cc)$, $p(y_j^\dd)$ are given by similar formulas with $a$ and $\e$ swapped.   One can observe that any atom at $\pm 1$ must have mass $0$. 
\end{definition}  
\begin{remark}
    Under an additional assumption that $ac<1$ if $c\in\RR$, the Askey--Wilson signed measures defined above become actual probability measures which have been studied in \cite{bryc2010askey}.
\end{remark}
\begin{remark}\label{rmk:atom weights are continuous functions}
One can observe that, for $\eee\in\{\aa,\bb,\cc,\dd\}$ both $y_j^\eee(a,b,c,d)$ and $p_j^\eee(a,b,c,d)$ for $j\geq 0$ are continuous functions on the subset of parameter region $(a,b,c,d)\in\Omega$ where atom $y_{j}^\eee$ exists (i.e. when the corresponding $\e\in\{a,b,c,d\}$ satisfies $|\e q^j|\geq 1$).  
\end{remark}

\subsection{Orthogonality}
\label{subsec:Orthogonality}
Under the condition \textbf{(1)} in Definition \ref{def:region omega} on $a,b,c,d$, \cite[Theorem 2.3]{askey1985some} shows that the Askey--Wilson polynomials satisfy orthogonality relations with respect to a measure on a specific contour $\C$.
When $(a,b,c,d)\in\WO$, this orthogonality can be written as Askey--Wilson polynomials being orthogonal with respect to the signed measure $\nu(\d x;a,b,c,d)$ given in Section \ref{subsec:Askey--Wilson measure}.
Using some continuity arguments, we will extend such orthogonality to the larger parameter region $\Omega$.
We recall from Section \ref{subsec:Askey--Wilson polynomials} that the Askey--Wilson polynomials are denoted by $w_m(x)=w_m(x;a,b,c,d|q)$ for $m\in\NN$.

\begin{theorem}[Theorem 2.3 in \cite{askey1985some}]
\label{thm: contour integral orthogonality Theorem 2.3 AW}
Assume $a,b,c,d\in\mathbb{C}$ satisfy condition \textbf{(1)} in Definition \ref{def:region omega}:
$$
a^2,b^2,c^2,d^2,ab,ac,ad,bc,bd,cd\notin\{q^{-l}:l\in\NN\}.
$$
Denote
$$
r(z)=\frac{(z^2,z^{-2})_\infty}{(az,a/z,bz,b/z,cz,c/z,dz,d/z)_\infty}.
$$
Assume $\C$ is a contour in $\mathbb{C}$ that encloses $\e q^l$ and excludes $\lb\e q^l\rb^{-1}$ for each $\e\in\{a,b,c,d\}$ and $l\in\NN$. 
Then for any $m,k\in\NN$,
\be\label{eq:AW contour integral}
\oint_\C\frac{\d z}{4\pi\i z}r(z)\tw_m\lb\frac{z+z^{-1}}{2}\rb\tw_k\lb\frac{z+z^{-1}}{2}\rb=\delta_{mk}\frac{(abcd)_\infty}{(q,ab,ac,ad,bc,bd,cd)_\infty}\frac{(1-q^{m-1}abcd)(q,ab,ac,ad,bc,bd,cd)_m}{(1-q^{2m-1}abcd)(abcd)_m}.
\ee
\end{theorem}

\begin{corollary}\label{cor:orthogonality AW on R}
Assume $a,b,c,d\in\WO$. Then for any $m,k\in\NN$,
\be \label{eq:orthogonality on R AW}
\int_\mathbb{R}\nu(\d x;a,b,c,d)\tw_m(x)\tw_k(x)=\delta_{mk}\frac{(1-q^{m-1}abcd)(q,ab,ac,ad,bc,bd,cd)_m}{(1-q^{2m-1}abcd)(abcd)_m}
\ee 
\end{corollary}
\begin{proof}
This corollary essentially follows from \cite[Theorem 2.5]{askey1985some}. Denote the integrand on LHS\eqref{eq:AW contour integral} as $s(z)$, which (by the definition of $\WO$) is a meromorphic function with singularity $z=0$ and simple poles at $z=\e q^j$ and $z=\lb\e q^j\rb^{-1}$, for any $j\in\NN$ and $\e\in\{a,b,c,d\}$. None of these poles lie on the unit circle $\{|z|=1\}$ or on $\C$. By considering the poles between $\C$ and $\{|z|=1\}$, one can write:

$$\text{LHS\eqref{eq:AW contour integral}}=\oint_{|z|=1}s(z)\d z+2\pi \i\sum_{\e,j:|\e q^j|>1}\res_{z=\e q^j}s(z)-2\pi \i\sum_{\e,j:|\e q^j|>1}\res_{z=\lb\e q^j\rb^{-1}}s(z),$$
where $\e\in\{a,b,c,d\}$ and $j\in\NN$. Since $r(z)=r(1/z)$, one has $\res_{z=\e q^j}s(z)+\res_{z=\lb\e q^j\rb^{-1}}s(z)=0$. The proof follows from \eqref{eq:AW contour integral} by computing both the continuous part integral and the residues.
\end{proof}
\begin{theorem}\label{thm:extending orthogonality}
    Assume $(a,b,c,d)\in\Omega$, then equation \eqref{eq:orthogonality on R AW}  still holds.
\end{theorem}
\begin{remark}
    This theorem covers all the cases in \cite{bryc2010askey} that will be useful for open ASEP, in particular cases $ a^2,b^2,c^2,d^2 \in\{q^{-l}:l\in\NN\}$, which are not explained clearly enough therein.
\end{remark}
\begin{proof}
When $(a,b,c,d)\in\WO$, we write \eqref{eq:orthogonality on R AW} explicitly: 
\be\label{eq: in the proof of mass 1 and orthogonality}\int_{-1}^1f(x;a,b,c,d)\tw_m(x)\tw_k(x)\d x+\sum_{x\in F(a,b,c,d)}p(x)\tw_m(x)\tw_k(x)=\delta_{mk}\frac{(1-q^{m-1}abcd)(q,ab,ac,ad,bc,bd,cd)_m}{(1-q^{2m-1}abcd)(abcd)_m}.\ee
We need to prove that the above equation holds for all $(a,b,c,d)\in\Omega\setminus\WO$.  

For any $(a_0,b_0,c_0,d_0)\in\Omega\setminus\WO$, we choose a sequence $\{(a_i,b_i,c_i,d_i)\}_{i=1}^\infty\subset\WO$ such that 
\begin{enumerate} 
    \item[\textbf{(i)}]   As $i\rightarrow\infty$, $(a_i,b_i,c_i,d_i)\rightarrow (a_0,b_0,c_0,d_0)$.
    \item[\textbf{(ii)}] $|a_i|\geq|a_0|$, $|b_i|\geq|b_0|$, $|c_i|\geq|c_0|$, $|d_i|\geq|d_0|$ for $i\in\mathbb{N}_+:=\{1,2,\ldots\}$. 
    \item[\textbf{(iii)}] $\# F(a_i,b_i,c_i,d_i)=\# F(a_0,b_0,c_0,d_0)$ for $i\in\mathbb{N}_+$.
    \item[\textbf{(iv)}] If $c_0,d_0$ are not real numbers, then $c_i=c_0$ and $d_i=d_0$ for $i\in\mathbb{N}_+$.
\end{enumerate}
We first show that such sequence exists. Since condition \textbf{(1)} in Definition \ref{def:region omega} fails, the product of some pair of $\{a_0,b_0,c_0,d_0\}$ equals to $q^{-l}$ for some $l\in\NN$. By condition \textbf{(4)} in Definition \ref{def:region omega} such pair must consists of two real numbers.
For $i\in \mathbb N_+$ and $\ep_i>0$ set
$$
(a_i,b_i,c_i,d_i)=\begin{cases}
    ((1+\ep_i)a_0,(1+\ep_i)b_0,(1+\ep_i)c_0,(1+\ep_i)d_0) &\text{if}\quad a_0,b_0,c_0,d_0\in\RR,\\
    ((1+\ep_i)a_0,(1+\ep_i)b_0,c_0,d_0)  &\text{if}\quad c_0,d_0\notin\RR.
\end{cases}
$$ 
Note that \textbf{(ii)} and \textbf{(iv)} are immediately satisfied. Under \textbf{(ii)}, condition \textbf{(iii)} can be guaranteed by $\ep_i$, $i\in\mathbb N_+$, being small enough. Since $(a_0,b_0,c_0,d_0)\in\Omega$, conditions  \textbf{(2)}, \textbf{(3)} and \textbf{(4)} in Definition \ref{def:region omega} for $(a_i,b_i,c_i,d_i)$ can also be guaranteed by $\ep_i$ being small enough, $i\in\mathbb N_+$.
One can then choose a sequence $\ep_i\rightarrow 0$ such that $(a_i,b_i,c_i,d_i)$ satisfy condition \textbf{(1)} in Definition \ref{def:region omega}. Condition \textbf{(i)} is then also satisfied.
 
Note that RHS of \eqref{eq: in the proof of mass 1 and orthogonality} is continuous on $\Omega$. We only need to prove that the LHS of \eqref{eq: in the proof of mass 1 and orthogonality} evaluated at $(a_i,b_i,c_i,d_i)$ converges as $i\rightarrow\infty$, to that evaluated at $(a_0,b_0,c_0,d_0)$.
We first look at the atomic part.
By \textbf{(ii)}, for any $\eee\in\{\aa,\bb,\cc,\dd\}$ and $j\in\NN$, either $y_j^\eee(a_i,b_i,c_i,d_i)\in F(a_i,b_i,c_i,d_i)$ for all $i\in\NN$ or 
$y_j^\eee(a_i,b_i,c_i,d_i)\notin F(a_i,b_i,c_i,d_i)$ for all $i\in\NN$.  
As $i\rightarrow\infty$, the positions $y_j^\eee(a_i,b_i,c_i,d_i)$ and masses $p_j^\eee(a_i,b_i,c_i,d_i)$ of these atoms converge respectively to $y_j^\eee(a_0,b_0,c_0,d_0)$ and $p_j^\eee(a_0,b_0,c_0,d_0)$. Since the Askey--Wilson polynomials $\tw_m(x)=\tw_m(x;a,b,c,d)$ are  continuous in $x,a,b,c,d$ (which follows from three-term recurrence \eqref{eq:three term recurrence}), we have: 
$$\lim_{i\rightarrow\infty}\sum_{x\in F(a_i,b_i,c_i,d_i)}p(x)\tw_m(x)\tw_k(x)=\sum_{x\in F(a_0,b_0,c_0,d_0)}p(x)\tw_m(x)\tw_k(x).$$

We then look at the continuous part $\int_{-1}^1f(x;a,b,c,d)\tw_m(x)\tw_k(x)\d x$. First, $\tw_m(x;a,b,c,d)\tw_k(x;a,b,c,d)$ is uniformly bounded for  $x\in(-1,1)$ and $(a,b,c,d)=(a_i,b_i,c_i,d_i)$, $i\in\NN$, and as $i\rightarrow\infty$, for any $x\in(-1,1)$, it converges to $\tw_m(x;a_0,b_0,c_0,d_0)\tw_k(x;a_0,b_0,c_0,d_0)$. Since $|1-e^{2\i\theta}|^2=4(1-x^2)$ we can write
\be\label{eq:in the proof of orthogonality continuous part}f(x;a,b,c,d)=\frac{(q, ab, ac, ad, bc, bd, cd)_{\infty}|(qe^{2\i\theta})_\infty|^2}{2\pi
(abcd)_{\infty}}\times\frac{4\sqrt{1-x^2}}{|(ae^{\i\theta},be^{\i\theta},ce^{\i\theta},de^{\i\theta})_\infty|^2}.\ee
The first fraction on the RHS of \eqref{eq:in the proof of orthogonality continuous part} is uniformly bounded on $x\in(-1,1)$ and $(a,b,c,d)=(a_i,b_i,c_i,d_i)$, $i\in\NN$, and as $i\rightarrow\infty$, for any fixed $x\in(-1,1)$, it converges to the one for $(a_0,b_0,c_0,d_0)$. 

We then look at the second fraction on the RHS of \eqref{eq:in the proof of orthogonality continuous part}.
Define $m_1$ (resp. $m_2$) to be the number of elements in $\{a_0,b_0,c_0,d_0\}$ that falls in $\{q^{-l}:l\in\NN\}$ (resp. $\{-q^{-l}:l\in\NN\}$). By condition \textbf{(3)} in Definition \ref{def:region omega}, $m_1,m_2\leq 1$.
When $m_1=m_2=1$, we have $\e_0,\f_0\in\{a_0,b_0,c_0,d_0\}$ such that $\e_0=q^{-r}$ and $\f_0=-q^{-s}$ for $r,s\in\NN$. We denote $\widetilde{\e}_i=q^r\e_i$ and $\widetilde{\f}_i=q^s\f_i$ for $i\in\NN$. 
We look at the denominator $\prod_{\mathfrak{h}_i\in\{a_i,b_i,c_i,d_i\}}\prod_{l=0}^{\infty}|1-q^l\mathfrak{h}_ie^{\i\theta}|^2$. Based on the assumption $\lim_{i\rightarrow\infty}(a_i,b_i,c_i,d_i)=(a_0,b_0,c_0,d_0)$, for sufficiently large $i$, all of the numbers $q^l\mathfrak{h}_i$ for $\mathfrak{h}_i\in\{a_i,b_i,c_i,d_i\}$ and $l\in\NN$ are uniformly bounded away from $1$, except for two: $q^r\e_i=\widetilde{\e}_i$ and $q^s\f_i=\widetilde{\f}_i$. Consequently, except for $|1-\widetilde{\mathfrak{e}}_ie^{\i\theta}|^2$ and $|1-\widetilde{\mathfrak{f}}_ie^{\i\theta}|^2$, all other terms in the denominator are uniformly bounded away from $0$. 
By \textbf{(ii)} we have $\widetilde{\e}_i\geq 1$ and  $\widetilde{\f}_i\leq -1$, hence $|1-\widetilde{\e}_ie^{\i\theta}|^2=1+\widetilde{\e}_i^2-2\widetilde{\e}_ix\geq2(1-x)$ and $|1-\widetilde{\f}_ie^{\i\theta}|^2=1+\widetilde{\f}_i^2-2\widetilde{\f}_ix\geq2(1+x)$.
Hence we have $\frac{4\sqrt{1-x^2}}{|1-\widetilde{\e}_ie^{\i\theta}|^2|1-\widetilde{\f}_ie^{\i\theta}|^2}\leq \frac{1}{\sqrt{(1-x)(1+x)}}$. Therefore $f(x;a,b,c,d)$ for $(a,b,c,d)=(a_i,b_i,c_i,d_i)$, $i\in\NN$ and $x\in(-1,1)$ are uniformly bounded by a constant times $\frac{1}{\sqrt{(1-x)(1+x)}}$, which is an integrable function  on $(-1,1)$. By the dominated convergence theorem 
\begin{multline*}
\lim_{i\rightarrow\infty}\int_{-1}^1f(x;a_i,b_i,c_i,d_i)\tw_m(x;a_i,b_i,c_i,d_i)\tw_k(x;a_i,b_i,c_i,d_i)\d x\\
=\int_{-1}^1f(x;a_0,b_0,c_0,d_0)\tw_m(x;a_0,b_0,c_0,d_0)\tw_k(x;a_0,b_0,c_0,d_0)\d x.
\end{multline*} 
In other cases $(m_1,m_2)=(1,0),(0,1)$ and $(0,0)$, one can similarly bound $f(x;a_i,b_i,c_i,d_i)$ by a constant times integrable functions $\frac{1}{\sqrt{1-x}}$,  $\frac{1}{\sqrt{1+x}}$ and $1$ and use the dominated convergence theorem.

By combining the discrete and continuous 
parts,
 we conclude that \eqref{eq: in the proof of mass 1 and orthogonality} holds for any $(a_0,b_0,c_0,d_0)\in\Omega\setminus\WO$, hence \eqref{eq: in the proof of mass 1 and orthogonality} holds for $(a,b,c,d)\in\Omega$.
\end{proof}


The following is a simple corollary of Theorem \ref{thm:extending orthogonality}:
\begin{corollary}\label{cor:1}
    Assume $(a,b,c,d)\in\Omega$. Then we have: 
\be\label{eq:mass equal 1}\int_\mathbb{R}\nu(\d x;a,b,c,d)=1.\ee
\end{corollary}

\begin{remark}\label{rmk:total variation uniformly bounded compact}
By combining \eqref{eq:mass equal 1} and Remark \ref{rmk:atom weights are continuous functions}, one can observe that the total variation of Askey--Wilson signed measure $\nu(\d x;a,b,c,d)$ is uniformly bounded on compact subsets of $\Omega$.
\end{remark}

\subsection{Multi-dimensional Askey--Wilson signed measures}
\label{subsec:The multi-time Askey--Wilson signed measures}
We will sequentially define three families of Askey--Wilson signed measures $\pi_t(\d y)$, $P_{s,t}(x,\d y)$ and $\pi_{t_1,\dots,t_n}(\d x_1,\dots,\d x_n)$ that will be useful in  studying open ASEP. 
When restricted to the fan region $AC<1$, they become probability measures, and are respectively the one-dimensional marginal laws, transitional probabilities and multi-dimensional marginal laws of the Askey--Wilson processes in \cite{bryc2010askey}. For general parameters they are finite signed measures with total mass $1$. 
Throughout this subsection we will assume $A,C\geq0$, $B,D\in(-1,0]$, and we also require an extra condition $ABCD\notin\{q^{-l}:l\in\NN\}$. 

On a suitable time interval $t\in I$, we will consider the Askey--Wilson signed measure:
\be\label{eq:marginal AW signed}\pi_t(\d y):=\nu\lb \d y;A\sqrt{t},B\sqrt{t},\frac{C}{\sqrt{t}},\frac{D}{\sqrt{t}}\rb.\ee 
We define 
$$U_t:=[-1,1]\cup F\lb A\sqrt{t},B\sqrt{t},C/\sqrt{t},D/\sqrt{t}\rb.$$
Note that $U_t$ is a compact subset of $\RR$ and $\pi_t$ is supported on $U_t$. For any $s,t\in I$, $s<t$  and $x\in U_s$ we will also consider the following  Askey--Wilson signed measure:
\be\label{eq:transition AW signed}P_{s,t}(x,\d y):=\nu\lb \d y;A\sqrt{t},B\sqrt{t},\sqrt{\frac{s}{t}}\lb x+\sqrt{x^2-1}\rb,\sqrt{\frac{s}{t}}\lb x-\sqrt{x^2-1}\rb\rb\ee 

We prove that there exists some time interval $I$ near $1$ such that the above signed measures  $\pi_t(\d y)$, $t\in I$, and $P_{s,t}(x,\d y)$, $s,t\in I$, $s<t$, $x\in U_s$,   are well-defined:
\begin{proposition}\label{prop:open ASEP parameters in Omega}
Assume  $A,C\geq0$, $B,D\in(-1,0]$ and $ABCD\notin\{q^{-l}:l\in\NN\}$,
then there exists an open interval $I$ containing $(1-\ep,1)$ for some $\ep>0$, such that 
for any $s,t\in I$, $s<t$ and $x\in U_s$, 
\be \label{eq: AW measure for ASEP in Omega}
\begin{split}
    &\lb A\sqrt{t},B\sqrt{t},\frac{C}{\sqrt{t}},\frac{D}{\sqrt{t}}\rb\in\Omega,\\
    &\lb A\sqrt{t},B\sqrt{t},\sqrt{\frac{s}{t}}\lb x+\sqrt{x^2-1}\rb,\sqrt{\frac{s}{t}}\lb x-\sqrt{x^2-1}\rb\rb\in\Omega.
\end{split}
\ee  
If we also assume 
 $A/C\notin\{q^l:l\in\mathbb{Z}\}$ if $A,C\geq 1$,
then one can chose open interval $I$ above such that $1\in I$. 

We denote the open interval $I$ chosen above by $I=I(A,B,C,D)$.
\end{proposition}
\begin{proof}
    Condition \textbf{(4)} in Definition \ref{def:region omega} clearly holds. Condition \textbf{(2)} in Definition \ref{def:region omega} follows since either $abcd=ABCD$ or $abcd=ABs/t\leq 0$.

We only need to chose $I$ to satisfy condition \textbf{(3)} in Definition \ref{def:region omega}.
 Let $I$ be an interval satisfying: 
\begin{enumerate} 
    \item[(i)]   For any $\tE\in\{A,B,C,D\}$, if $|\tE|<1$, we have $|\tE\sqrt{t}|,|\tE/\sqrt{t}|<1$ for any $t\in I$.
    \item[(ii)]   For any $\tE\in\{A,B,C,D\}$, if $|\tE|>1$, we have $|\tE\sqrt{t}|,|\tE/\sqrt{t}|>1$ for any $t\in I$.
    \item[(iii)] $I\subset (\sqrt{q},1/\sqrt{q})$. Hence $s/t\notin\{q^l:l\in\mathbb{Z}\}$ for any $s<t$ in $I$.
    \item[(iv)]  $At/C\notin\{q^l:l\in\mathbb{Z}\}$ for all $t\in I$, if $A,C\geq 1$.
\end{enumerate}
The first two conditions can be guaranteed by $I$ being a subset of a small enough neighbourhood of $1$. The third condition in general holds on $t\in(1-\ep,1)$ for some $\ep>0$, and also holds on $t\in(1-\ep,1+\ep)$ for some $\ep>0$, if one furthermore assume that $A/C\notin\{q^l:l\in\mathbb{Z}\}$ if $A,C\geq 1$. 

We first look at $\lb A\sqrt{t},B\sqrt{t},C/\sqrt{t},D/\sqrt{t}\rb$. For $t\in I$, in view of (i)  we have $|B\sqrt{t}|<1$ and $|D/\sqrt{t}|<1$.
If $A<1$ or $C<1$ then, similarly, for $t\in I$ we have $A\sqrt{t}<1$ or $C/\sqrt{t}<1$, and  thus condition \textbf{(3)} in Definition \ref{def:region omega} holds. If $A,C\geq 1$ we have $\left|\frac{A\sqrt{t}}{C/\sqrt{t}}\right|=At/C$ for $t\in I$,  so  in view of (iv), condition \textbf{(3)} in Definition \ref{def:region omega} holds.

We then look at $\lb A\sqrt{t},B\sqrt{t},\sqrt{\frac{s}{t}}\lb x+\sqrt{x^2-1}\rb,\sqrt{\frac{s}{t}}\lb x-\sqrt{x^2-1}\rb\rb$ for $x\in U_s$. In view of (i), we have $|B\sqrt{t}|<1$, and also $U_s$ does not have atoms $<-1$.
\begin{itemize}
    \item [Case 1.]  Let $x\in[-1,1]$. Then $\sqrt{\frac{s}{t}}\lb x+\sqrt{x^2-1}\rb$ and $\sqrt{\frac{s}{t}}\lb x-\sqrt{x^2-1}\rb$ are complex conjugate pairs with norm $<1$ and condition \textbf{(3)} in Definition \ref{def:region omega} holds. 
    \item [Case 2.] Let $x>1$ and $A<1$. In view of (i), we have $|A\sqrt{t}|<1$. Note that $\sqrt{\frac{s}{t}}\lb x+\sqrt{x^2-1}\rb$ and $\sqrt{\frac{s}{t}}\lb x-\sqrt{x^2-1}\rb$ are positive numbers, one of which has norm $<1$. Therefore, condition \textbf{(3)} in Definition \ref{def:region omega} never applies and hence vacuously holds.
    \item [Case 3.] Let $x>1$ and $A\geq 1$. Since $x\in U_s$, either $x=\frac{1}{2}\lb q^jA\sqrt{s} +(q^jA\sqrt{s})^{-1}\rb$ (in case $q^jA\sqrt{s}>1$) or $x=\frac{1}{2}\lb q^jC/\sqrt{s}+(q^jC/\sqrt{s})^{-1}\rb$ (in case $q^jC/\sqrt{s}>1$). In the first case $\sqrt{\frac{s}{t}}\lb x+\sqrt{x^2-1}\rb=Asq^j/\sqrt{t}$ and the ratio with $A\sqrt{t}$ equals $q^js/t$. So condition \textbf{(3)} in Definition \ref{def:region omega} holds by (iii). In the second case we have $A,C\geq 1$, and $\sqrt{\frac{s}{t}}\lb x+\sqrt{x^2-1}\rb=Cq^j/\sqrt{t}$ whose ratio with $A\sqrt{t}$ equals $q^jC/(At)$.  So condition \textbf{(3)} in Definition \ref{def:region omega} holds by (iv).
\end{itemize} 
\end{proof} 
\begin{definition}
     Assume  $A,C\geq0$, $B,D\in(-1,0]$, $ABCD\notin\{q^{-l}:l\in\NN\}$ and $I=I(A,B,C,D)$.
    For any $s,t\in I$, $s<t$ and $x\in U_s$, we define the finite signed measures $\pi_t(\d y)$ and $P_{s,t}(x,\d y)$ respectively by \eqref{eq:marginal AW signed} and \eqref{eq:transition AW signed}.
    For convenience, when $s,t\in I$, $s\leq t$  and $x\notin U_s$ we define $P_{s,t}(x,\d y)=0$.
    When $s=t\in I$ and $x\in U_s$ we define $P_{s,s}(x,\d y)=\delta_x(\d y)$. 
\end{definition}
\begin{remark}\label{rmk:total mass 1}
    One can observe from \eqref{eq:mass equal 1} that the finite signed measures $\pi_t(\d y)$ and $P_{s,t}(x,\d y)$ have total mass $1$, i.e. $\int_{\RR}\pi_t(\d y)=1$ and $\int_{\RR}P_{s,t}(x,\d y)=1$ for any $s,t\in I$, $s\leq t$ and $x\in U_s$.
\end{remark}
\begin{lemma}\label{lem:support of signed measure Pst}
Assume  $A,C\geq0$, $B,D\in(-1,0]$, $ABCD\notin\{q^{-l}:l\in\NN\}$ and $I=I(A,B,C,D)$. Then for any $s,t\in I$, $s\leq t$,  the signed measure $P_{s,t}(x,\d y)$ is supported on $U_t$.
\end{lemma}
\begin{proof}
    We only need to prove the case when $s<t$ and $x\in U_s$. We first show that, if $x\in U_s$ is an atom generated by $A\sqrt{s}$ or $B\sqrt{s}$, then the atoms of $P_{s,t}(x,\d y)$ generated by $\sqrt{\frac{s}{t}}\lb x\pm\sqrt{x^2-1}\rb$ have mass $0$.
    Assume $\tE\in\{A,B\}$ and $x=\frac{1}{2}\lb\tE\sqrt{s}q^j+(\tE\sqrt{s}q^j)^{-1}\rb$ for some $j\in\NN$ such that $|\tE\sqrt{s}q^j|\geq 1$. Then
    $$P_{s,t}(x,\d y)=\nu\lb \d y;A\sqrt{t},B\sqrt{t},(\tE q^j)s/\sqrt{t},(\tE q^j)^{-1}/\sqrt{t}\rb.$$
    We have $|(\tE q^j)^{-1}/\sqrt{t}|<|\tE\sqrt{s}q^j|^{-1}\leq 1$, hence $(\tE q^j)^{-1}/\sqrt{t}$ does not generate atoms. For the atoms generated by  $(\tE q^j)s/\sqrt{t}$, one has $(\tE\sqrt{t})\lb(\tE q^j)^{-1}/\sqrt{t}\rb=q^{-j}$. This means either $ad=q^{-j}$ or $bd=q^{-j}$ in $P_{s,t}(x,\d y)=\nu(\d y;a,b,c,d)$. Using the formulas \eqref{eq: p_0} and \eqref{eq: p_j} for atom masses (in which we need to swap $a$ with $c$ to get masses for atoms generated by $\cc$), we have $p(x_{k}^\cc)=0$ for all atoms $x_{k}^\cc$ generated by $c=(\tE q^j)s/\sqrt{t}$.
    

    We return to the proof. We split it into the following cases:
    \begin{enumerate}
        \item[Case 1.] If $|x|<1$, we have $\vert\sqrt{\frac{s}{t}}\lb x\pm\sqrt{x^2-1}\rb\vert<1$ which do not generate atom, hence all the atoms of $P_{s,t}(x,\d y)$ are generated by $A\sqrt{t}$ and $B\sqrt{t}$, and they are contained in $U_t$.
        \item[Case 2.] If $x\in U_s$ is an atom generated by $A\sqrt{s}$ or $B\sqrt{s}$, then by the argument above,
    the atoms of $P_{s,t}(x,\d y)$ with nonzero mass are necessarily generated by $A\sqrt{t}$ and $B\sqrt{t}$, which are contained in $U_t$.
    \item[Case 3.] If $x\in U_s$ is an atom generated by $C/\sqrt{s}$ or $D/\sqrt{s}$, one can write $x=\frac{1}{2}\lb\tE q^j/\sqrt{s}+\sqrt{s}/(\tE q^j)\rb$, where $\tE\in\{C,D\}$, $j\in\NN$ such that $|\tE q^j/\sqrt{s}|\geq1$. Then $\vert\sqrt{\frac{s}{t}}\lb x-\sqrt{x^2-1}\rb\vert=\vert sq^{-j}/(\tE\sqrt{t})\vert<\vert\sqrt{s}/(\tE q^j)\vert\leq 1$, hence it does not generate atoms. On the other hand $\sqrt{\frac{s}{t}}\lb x+\sqrt{x^2-1}\rb=\tE q^j/\sqrt{t}$, and when $|\tE q^j/\sqrt{t}|\geq1$ it generates atoms which  are contained in $U_t$. Other atoms can be generated by $A\sqrt{t}$ and $B\sqrt{t}$,  and thus  they are contained in $U_t$.
    \end{enumerate} 

    This concludes the proof that the signed measure $P_{s,t}(x,\d y)$ is supported on $U_t$.
\end{proof}

Next is a special property of the Askey--Wilson signed measure $P_{s,t}(x,\d y)$ that will be useful later.
\begin{lemma}\label{lem:property on support}
    Assume $A,C\geq0$, $B,D\in(-1,0]$, $ABCD\notin\{q^{-l}:l\in\NN\}$,  and $I=I(A,B,C,D)$. Assume furthermore that $A>1$, then for $t\in I$, $A\sqrt{t}$ generates a set of atoms, and we denote the largest among them by  
     $y_0^{\aa}(t):=\frac{1}{2}\lb A\sqrt{t}+\lb A\sqrt{t}\rb^{-1}\rb\in U_t$. 
     Then for any $s,t\in I$, $s\leq t$, we have $P_{s,t}\lb y_0^{\aa}(s),\d y\rb=\delta_{y_0^{\aa}(t)}\lb\d y\rb$.  
\end{lemma}
\begin{proof} We write: 
    $$P_{s,t}\lb y_0^\aa(s),\d y\rb=\nu\lb \d y;A\sqrt{t},B\sqrt{t},As/\sqrt{t},
    1/(A\sqrt{t}) \rb.$$
    The result can then be observed from Definition \ref{def:definition of Askey--Wilson measure}.
\end{proof}


\begin{definition}
    Assume  $A,C\geq0$, $B,D\in(-1,0]$, $ABCD\notin\{q^{-l}:l\in\NN\}$ and $I=I(A,B,C,D)$. 
    For any $t_1,\ldots,t_n\in I$, satisfying $t_1\leq\dots\leq t_n$, we define a finite signed measure $\pi_{t_1,\dots,t_n}$ supported on $U_{t_1}\times\dots\times U_{t_n}$  by 
    \be\label{eq:definition of multi-time}\pi_{t_1,\dots,t_n}(\d x_1,\dots,\d x_n):=\pi_{t_1}(\d x_1)P_{t_1,t_2}(x_1,\d x_2)\dots P_{t_{n-1},t_n}(x_{n -1},\d x_n).\ee
    \end{definition}
  \begin{remark} Note that \eqref{eq:definition of multi-time} gives a well-defined finite signed measure. Indeed, let us consider the following linear functional on $C(U_{t_1}\times\dots\times U_{t_n})$ (which is the space of continuous functions with supremum norm):
    \be\label{eq:functional defining multi-time measure}
     f\mapsto\int_{U_{t_1}}\pi_{t_1}(\d x_1)\int_{U_{t_2}}P_{t_1,t_2}(x_1,\d x_2)\dots\int_{U_{t_n}}P_{t_{n-1},t_n}(x_{n-1},\d x_n)f(x_1,\dots,x_n).\ee
    By Remark \ref{rmk:total variation uniformly bounded compact},  for any fixed $s,t\in I$, $s\leq t$, the total variation of $P_{s,t}(x,\d y)$ is  uniformly bounded by a finite constant independent of $x\in\RR$. Hence,  \eqref{eq:functional defining multi-time measure} defines a bounded linear functional on $C(U_{t_1}\times\dots\times U_{t_n})$, which, by  the Riesz representation theorem (see for example \cite[Theorem 6.19]{rudin1987real}), is the integration  with respect to a unique finite signed measure $\pi_{t_1,\dots,t_n}$ that was defined in \eqref{eq:definition of multi-time}. 
    
    One can observe from Remark \ref{rmk:total mass 1} that $\pi_{t_1,\dots,t_n}$ has total mass $1$, i.e. $\int_{\RR^n}\pi_{t_1,\dots,t_n}(\d x_1,\dots,\d x_n)=1$.
\end{remark}

\section{Proof of Theorem \ref{thm: stationary measure in terms of Askey--Wilson integral}}
\label{sec:Proof of Theorem stationary measure in terms of Askey--Wilson integral}
The proof follows a similar procedure as \cite[Section 2.1]{bryc2017asymmetric}. We first extend the `projection formula' in  \cite{bryc2010askey} to a larger parameter range involving Askey--Wilson signed measures. By combining the USW representation of the DEHP algebra (see Section \ref{subsubsec:USW representation} for a brief review), the orthogonality of $\pi_t(\d x)$, the projection formula of $P_{s,t}(x,\d y)$ and three-term recurrence of Askey--Wilson polynomials, we are able to peel off the linear terms in  $\ll W|(\E+t_1\D)\times\dots\times(\E+t_n\D)|V\rr$ one by one, from left to right.

We will use the re-normalized version of Askey--Wilson polynomials as in \cite{bryc2010askey}:
\begin{definition}\label{def:bw polynomials}
    Assume $a,b,c,d\in\mathbb{C}$ and $ab, abcd\notin\{q^{-l}:l\in\NN\}$. We define $\bw_m(x)=(ab)_m^{-1}\tw_m(x)$ for $m\in\NN$, which are also called Askey--Wilson polynomials. Note that $\bw_0(x)=1$.  
\end{definition} 

The following orthogonality of Askey--Wilson polynomials is a simple corollary of Theorem \ref{thm:extending orthogonality}:
\begin{corollary}\label{cor:orthogonality for bar w}
    Assume $(a,b,c,d)\in\Omega$.  Then we have: 
$$\int_\mathbb{R}\nu(\d x;a,b,c,d)\bw_m(x)\bw_k(x)=0 \quad\text{for all } m,k\in\NN, m\neq k.$$
\be\label{eq:orthogonality of bar w}
\int_\mathbb{R}\nu(\d x;a,b,c,d)\bw_m(x)=0\quad\text{for all } m\in\mathbb{N}_+.
\ee
\end{corollary} 

\subsection{Projection formula}
A projection formula for Askey--Wilson polynomials was introduced in \cite{nassrallah1985projection} and later generalized in \cite{bryc2010askey}. We show that \cite[Proposition 3.6]{bryc2010askey} can be extended to a larger parameter range where $P_{s,t}(x,\d y)$ are no longer probability measures, but signed measures.
\begin{proposition}\label{prop:projection formula}
Assume $A,C\geq0$, $B,D\in(-1,0]$, $ABCD\notin\{q^{-l}:l\in\NN\}$, and  $I=I(A,B,C,D)$.
For $n\in\NN$, we consider polynomials :
$$p_m(x;t)=t^{m/2}\bw_m(x;A\sqrt{t},B\sqrt{t},C/\sqrt{t},D/\sqrt{t}),$$
where $\bw_m(x;a,b,c,d)$ is introduced in Definition \ref{def:bw polynomials}. 

Then for any $s,t\in I$, $s\leq t$, $x\in U_s$ and $m\in\NN$,
\be\label{eq:projection formula}\int_\mathbb{R}p_m(y;t)\bP_{s,t}(x,\d y)=p_m(x;s).\ee 
\end{proposition}
\begin{proof}
We only need to prove case $s<t$.
As in \cite{bryc2010askey}, we introduce a family of polynomials  $(Q_m(\cdot;x,t,s))_{m=1}^{\infty}$: For $m\in\NN$:
    \be\label{eq:def of Q}Q_m(y;x,t,s):=t^{m/2} \bw_m\lb y;A\sqrt{t},B\sqrt{t},\sqrt{\frac{s}{t}} \lb x+\sqrt{x^2-1} \rb, \sqrt{\frac{s}{t}} \lb x-\sqrt{x^2-1} \rb\rb.\ee
We first prove an algebraic lemma: 
\begin{lemma}\label{lem:algid}
    For $m\geq1$, 
\be \label{eq:connection formula algebraic identity}
Q_m(y;x,t,s)=\sum_{r=1}^m b_{m,r}(x,s) \lb p_r(y;t)-p_r(x;s) \rb,
\ee
where $b_{m,r}(x,s)$ does not depend on $t$ for $1\leq r\leq m$, additionally, 
$b_{m,m}(x,s)$ does not depend on $x$, and $b_{m,m}(x,s)\ne0$.\end{lemma}
     
\begin{proof}[Proof of Lemma \ref{lem:algid}]
We begin by recalling \cite[Theorem A.2]{bryc2010askey} which re-states a special case of \cite[formula (6.1)]{askey1985some}. For any $a,b,c,d,\tc,\td\in\mathbb{C}$, $a\neq 0$ and $ab, abcd\notin\{q^{-l}:l\in\NN\}$, we have, for $m\in\NN$:
\be\label{eq:connection formula}
\bw_m(y;a,b,\widetilde c,\widetilde d)=\sum_{r=0}^m \bc_{r,m}
\bw_r(y; a,b,c,d),
\ee
where 
 $$
\bc_{r,m} =  (-1)^r q^{r(r+1)/2} \times\frac{ (q^{-m},q^{m-1}a b \widetilde c \widetilde d
)_r(a\widetilde c,a\widetilde d)_{m}}{a^{m-r}
(q,q^{r-1}abcd,a\widetilde c,a\widetilde d )_r} 
\times{}_4\phi_3\left[ \genfrac{}{}{0pt}{}{q^{r-m},ab\widetilde{c}\widetilde{d}q^{m+r-1},acq^r,adq^r}        {abcdq^{2r},a\widetilde{c}q^r,a\widetilde{d}q^r},q \right].
 $$
Here we use the basic hypergeometric function:
$$
_r\phi_s\left[ \genfrac{}{}{0pt}{}{a_1,\dots,a_r}        {b_1,\dots,b_s},z\right]={_r\phi_s}\left[ \genfrac{}{}{0pt}{}{a_1,\dots,a_r}        {b_1,\dots,b_s};q,z\right]=\sum_{k=0}^{\infty}\frac{(a_1,\dots,a_r;q)_k}{(b_1,\dots,b_s,q;q)_k}\lb(-1)^kq^{k(k-1)/2}\rb^{1+s-r}z^k.
 $$
Inserting 
$$a=A\sqrt{t},\quad b=B\sqrt{t},\quad c=C/\sqrt{t},\quad d=D/\sqrt{t},\quad \tc=\sqrt{\frac{s}{t}} \lb x+\sqrt{x^2-1} \rb,\quad \td=\sqrt{\frac{s}{t}} \lb x-\sqrt{x^2-1} \rb$$
 into \eqref{eq:connection formula} we get 
 \be\label{eq:in proof of connection formula}
 Q_m(y;x,t,s)=\sum_{r=0}^m b_{m,r} p_r(y;t),\ee 
 where $b_{m,r}=t^{(m-r)/2}\bc_{r,m}$. 
Coefficients $b_{m,r}$ do not depend on $t$ as
$t^{(m-r)/2}/a^{m-r}=A^{r-m}$, and $t$ cancels out in all other entries
on the right-hand side of $\bc_{r,m}$ since 
\begin{multline*}
ab\widetilde c\widetilde d = ABs,\quad abcd=ABCD, \quad
a\widetilde c = A\sqrt{s}\lb x+\sqrt{x^2-1}\rb,\quad ac=AC, \\ \quad
a\widetilde d = A\sqrt{s}\lb x-\sqrt{x^2-1}\rb,\quad a d=AD. 
\end{multline*}
 Moreover, 
$
b_{m,m}(x,s)=(-1)^m q^{m(m+1)/2}\frac
{(q^{-m},q^{m-1}ABs)_m}{(q,q^{m-1}ABCD)_m}
$
and thus it is nonzero and does not depend on $x$. 

We note that $\tc\td=1$. By an induction using three-term recurrence \eqref{eq:three term recurrence} we get $Q_m(x;x,s,s)=0$ for $m\geq 1$. Referring to   \eqref{eq:in proof of connection formula} with $y=x$ and $t=s$ we get $\sum_{r=0}^m b_{m,r}(x,s) p_r(x;s)=0$.
Subtracting the latter from \eqref{eq:in proof of connection formula} concludes the proof. 
\end{proof}

 Now we are ready to prove \eqref{eq:projection formula}. We use induction with respect to $m$. 
    The $m=0$ case follows from \eqref{eq:mass equal 1}. Suppose that \eqref{eq:projection formula}  holds  for some $m\geq0$. In view of definitions \eqref{eq:transition AW signed} and \eqref{eq:def of Q}, using the orthogonality \eqref{eq:orthogonality of bar w}, we have \be\label{eq:orthogonality bw in the proof}\int_\mathbb{R}Q_{m+1}(y;x,t,s) \bP_{s,t}(x,\d y)=0.\ee
By the identity \eqref{eq:connection formula algebraic identity}, 
$Q_{m+1}(y;x,t,s)=\sum_{r=1}^{m+1} b_{m+1,r}(x,s) \lb p_r(y;t)-p_r(x;s) \rb$. We plug this into \eqref{eq:orthogonality bw in the proof}. Thus using the induction hypothesis, we obtain
\begin{equation*}
    \begin{split}
        0&=b_{m+1,m+1}(x,s)\int_\mathbb{R}\lb p_{m+1}(y;t)-p_{m+1}(x;s) \rb
\bP_{s,t}(x,\d y)\\
&=b_{m+1,m+1}(x,s) \lb\int_\mathbb{R}p_{m+1}(y;t) \bP_{s,t}(x,\d y)
-p_{m+1}(x;s) \rb.
    \end{split}
\end{equation*} 
Since $b_{m+1,m+1}(x,s)\neq 0$, we see that  \eqref{eq:projection formula}   holds  for  $m+1$ case which ends  the proof.
\end{proof}
 
\subsection{USW representation}\label{subsubsec:USW representation}
A representation of the $\dehp$ algebra for general parameters $q,\alpha,\beta,\gamma,\delta$ was introduced in \cite{uchiyama2004asymmetric}, called the $\usw$ representation. We   review a slightly different version from  \cite{bryc2017asymmetric}.

Assume $\alpha,\gamma>0$, $\beta,\delta\in(-1,0]$ and $q\in[0,1)$. We recall that $A,B,C,D$ are given by \eqref{eq:defining ABCD} and that $A, C\geq0$ and $B,D\in(-1,0]$. Additionally we assume that $ABCD\notin\{q^{-l}:l\in\NN\}$. For $m\in\NN$, we define $\alpha_m,\beta_m,\gamma_m,\delta_m,\ep_{m},\varphi_{m}$ in terms of $(A,B,C,D,q)$ by the formulas given in \cite[page 1243]{bryc2010askey}:
\be 
\begin{split}
    \alpha_m&=-AB q^m \beta_m ,\\
\beta_m&=\frac{1-A B C D q^{m-1} }
{\sqrt{1-q} (1-A B C D q^{2 m} ) (1-A B
C D q^{2 m-1} )} ,\\
\varepsilon_m&=\frac{(1-q^m)
(1-A C
q^{m-1} ) (1-AD
q^{m-1} ) (1-B C
q^{m-1} ) (1-B D q^{m-1} )
}{\sqrt{1-q} (1-A B C D q^{2
m-2} ) (1-A B C D q^{2 m-1} )} ,\\
\varphi_m&=-CD q^{m-1}\varepsilon_m ,\\
\gamma_m&= \frac{A}{\sqrt{1-q}}-\frac{\alpha_m}{A} (1-AC q^m)(1-ADq^m)-
\frac{A\varepsilon_m}{(1-AC q^{m-1})(1-ADq^{m-1})},\\
\delta_m&=\frac{1}{A\sqrt{1-q}}-\frac{\beta_m}{A} (1-AC q^m)(1-ADq^m)-
\frac{A\varphi_m}{(1-AC q^{m-1})(1-ADq^{m-1})}.
\end{split}
\ee 
We note that $\gamma_m$ and $\delta_m$ above are well-defined by the choices of $\ep_m$ and $\varphi_m$.
These formulas are also well-defined for $q=0$ and/or $A=0$ by continuity.

Consider  infinite  tridiagonal matrices:
$$\E=\frac{1}{1-q}\mathbf{I}+\frac{1}{\sqrt{1-q}}\mathbf{y},\quad \D=\frac{1}{1-q}\mathbf{I}+\frac{1}{\sqrt{1-q}}\mathbf{x},$$
where $\bI$ denotes the infinite identity matrix,  
\begin{equation*}
      \mathbf{x}=\left[\begin{matrix}
        \gamma_0 & \ep_1 & 0 &\dots \\
        \alpha_0 &  \gamma_1& \ep_2 &\dots \\
        0 & \alpha_1 & \gamma_2& \dots\\
        \vdots &\vdots & \vdots& \ddots\\
      \end{matrix}\right], \quad \mathbf{y}=\left[\begin{matrix}
        \delta_0 & \varphi_1 & 0 &\dots \\
        \beta_0 & \delta_1 & \varphi_2 & \dots\\
        0 & \beta_1 & \delta_2 & \dots\\
        \vdots & \vdots& \vdots& \ddots\\
      \end{matrix}\right]\,
    \end{equation*}
and  infinite vectors 
\be\label{eq:W and V}\ll W|=(1,0,0,\dots),\quad |V\rr=(1,0,0,\dots)^T.\ee 
As proved in \cite[Section 2.1]{bryc2017asymmetric}, they satisfy  conditions of  the $\dehp$ algebra \eqref{eq:DEHP algebra} with parameters $(q,\alpha,\beta,\gamma,\delta)$.

\subsection{Proof of Theorem \ref{thm: stationary measure in terms of Askey--Wilson integral}}
We define the polynomial $\Pi_n$ by 
$$
\Pi_n\lb t_1,\dots,t_n\rb=(1-q)^n\ll W|(\E+t_1\D)\times\dots\times(\E+t_n\D)|V\rr,
$$
where $\D$, $\E$, $\ll W|$ and $|V\rr$ are specified above. By the matrices $\D$ and $\E$ being tridiagonal and the specific forms \eqref{eq:W and V} of $\ll W|$ and $|V\rr$, one could observe that $\Pi_n$ is actually a polynomial in $t_1,\dots,t_n$ with coefficients depending on $A,B,C,D,q$. As in \cite{derrida1993exact} (see \eqref{eq:MPA open ASEP}), since $\D$, $\E$, $\ll W|$ and $|V\rr$ satisfy the $\dehp$ algebra \eqref{eq:DEHP algebra}, the joint generating function of  stationary measure can be written in terms of $\Pi_n$ as \eqref{eq:MPA open ASEP1} in Theorem \ref{thm: stationary measure in terms of Askey--Wilson integral}. The denominator of \eqref{eq:MPA open ASEP}  (and of \eqref{eq:MPA open ASEP1})  being nonzero can be  guaranteed by our assumption $ABCD\notin\{q^{-l}:l\in\NN\}$, following the arguments in \cite[Appendix A]{mallick1997finite} (see also \cite[Remark 3]{bryc2019matrix}).

Next we will prove the characterization \eqref{PIN} of $\Pi_n$ as integrations with Askey--Wilson signed measures. 
We take $I=I(A,B,C,D)$ as in Proposition \ref{prop:open ASEP parameters in Omega}. In the following we always assume $t_1\leq\dots\leq t_n$ in $I$. 

 We introduce the row vector of polynomials 
    $$\ll\p_t(x)|:=(p_0(x,t),p_1(x,t),\dots).$$
    By Corollary \ref{cor:1} and Corollary \ref{cor:orthogonality for bar w}, we have $$\int_\RR\pi_t(\d x)\ll\p_t(x)|=\ll W|.$$ By Proposition \ref{prop:projection formula}, for $x\in U_s$ we have $$\int_\RR P_{s,t}(x,\d y)\ll\p_t(y)|=\ll\p_s(x)|.$$  
    Clearly, $\ll\p_t(x)|V\rr=1$. The three-term recurrence \eqref{eq:three term recurrence} can also be written in the vector form:
    \be\label{eq:vec}\frac{2\sqrt{t}}{\sqrt{1-q}}x\ll\p_t(x)|=\ll\p_t(x)|(t\x+\y),\ee whence  we have  $$\ll\p_t(x)|\lb(1+t)\bI+\sqrt{1-q}(t\x+\y)\rb=\lb1+t+2\sqrt{t}x\rb\ll\p_t(x)|.$$
We remark that \eqref{eq:vec} coincides with
\cite[equation (1.16)]{bryc2017asymmetric}
modulo the transformation $\ll\p_t(x)|=\Big{\ll}\mathbf{r}_t\lb \frac{2\sqrt{t}}{\sqrt{1-q}}x\rb\Big{|}$.
    
    Using the above relations, we obtain:
    \begin{equation*}
        \begin{split}
            \Pi_n(t_1,\dots,t_n)&=\int_\RR\pi_{t_1}(\d x_1)\ll\p_{t_1}(x_1)|\prod_{i=1}^{\substack{n\\\longrightarrow}}\lb(1+t_i)\bI+\sqrt{1-q}(t_i\x+\y)\rb|V\rr\\
            &=\int_\RR\pi_{t_1}(\d x_1)\lb1+t_1+2\sqrt{t_1}x_1\rb\ll\p_{t_1}(x_1)|\prod_{i=2}^{\substack{n\\\longrightarrow}}\lb(1+t_i)\bI+\sqrt{1-q}(t_i\x+\y)\rb|V\rr\\
            &=\int_\RR\pi_{t_1}(\d x_1)\lb1+t_1+2\sqrt{t_1}x_1\rb
            \int_\RR P_{t_1,t_2}(x_1,\d x_2)
            \ll\p_{t_2}(x_2)|\prod_{i=2}^{\substack{n\\\longrightarrow}}\lb(1+t_i)\bI+\sqrt{1-q}(t_i\x+\y)\rb|V\rr\\
            &=\dots\\
            &=\int_{\RR^n}\prod_{i=1}^n(1+t_i+2\sqrt{t_i}x_i)\bpi_{t_1}(\d x_1)\bP_{t_1,t_2}(x_1,\d x_2)\dots \bP_{t_{n-1},t_n}(x_{n-1},\d x_n)\ll\p_{t_n}(x_n)|V\rr\\
            &=\int_{\RR^n}\prod_{i=1}^n(1+t_i+2\sqrt{t_i}x_i)\pi_{t_1,\dots,t_n}(\d x_1,\dots,\d x_n),
        \end{split}
    \end{equation*}
 which concludes the proof.
\begin{remark}
    We remark that the {\em formal} Chapman–Kolmogorov equation continues to hold for the signed measures $P_{s,t}(x,\d y)$. It is not needed in this paper so we omit the details. Moreover, we point out that the Markov property of the Askey--Wilson processes (i.e. the Chapman–Kolmogorov equation) is {\em not} actually needed in the exploitation of this method in the fan region in \cite{bryc2017asymmetric,bryc2019limit,bryc2022asymmetric}. 
\end{remark} 
 
\section{Limit fluctuations in high/low density phases: Proof of Theorem \ref{thm:fluctuations}}
\label{sec:limit fluctuation}
In this section we will prove Theorem \ref{thm:fluctuations}. The proof has a similar structure as \cite[Section 3]{bryc2019limit} in the fan region, but since we now have Askey--Wilson signed measures, we will need a bound on their total variations given in Appendix \ref{sec:total variation bounds}. 
We will first prove the result in the high density phase, and in the low density phase, it follows by the particle-hole duality. 
\subsection{Proof in the high density phase}
In this subsection we prove Theorem \ref{thm:fluctuations} in the high density phase 
\be\label{eq:high density phase condition}A,C\geq0,\quad B,D\in(-1,0],\quad A>C,\quad A>1.\ee 
We first  adopt a method similar to \cite[Section 3.1]{bryc2019limit} to prove the limit fluctuations under a technical constraint. Then we extend the result to the whole high density phase adopting the `stochastic sandwiching' argument from \cite[Lemma 5.1]{corwin2021stationary}. 

\subsubsection{Proof in a generic sub-region}
In this step we prove Theorem \ref{thm:fluctuations} in a generic sub-region of the high density phase \eqref{eq:high density phase condition} specified by $A/C,ABCD\notin\{q^{-l}:l\in\NN\}$. 

Using \cite[Theorem A.1]{bryc2019limit} (see  earlier works \cite{farrell2006techniques,hoffman2017probability,mukherjea2006note}), one can reduce the proof  of convergence in finite dimensional distribution to the convergence of the  multi-point Laplace transform: 
\begin{proposition}[Theorem A.1 in \cite{bryc2019limit}]
\label{thm:Theorem A.1 from bryc wang}
    Let $\boldsymbol{X}^{(n)}=\lb X_1^{(n)},\dots,X_d^{(n)}\rb$ be a sequence of random variables with Laplace transform 
    $$L_n(\boldsymbol{z})=L_n(z_1,\dots,z_d)=\mathbb{E}\exp\lb\sum_{i=1}^nz_iX_i^{(n)}\rb.$$
     Assume that on an open subset of $\RR^d$ Laplace transforms $L_n$ are finite and converge point-wise to a function $L$.  If on this open subset,  $L$  is the Laplace transform of a random variable $\boldsymbol{Y}=\lb Y_1,\dots,Y_d\rb$, then $\boldsymbol{X}^{(n)}$ converges in distribution to $\boldsymbol{Y}$.
\end{proposition}

We first recall the Laplace transform of the Brownian motion.
For $x_0=0<x_1<\dots<x_d=1$ and variables
 $c_1,\dots,c_d>0$, we denote $s_k = c_k+\cdots+c_d$ for $k=1,\dots,d$. We have:
$$
\mathbb{E}\le \exp\lb-\sum_{k=1}^d{c_k}\B({x_k})\rb\re= \mathbb{E}\le\exp\lb-\sum_{k=1}^d{s_k}(\B({x_k})-\B({x_{k-1}}))\rb\re = \exp\lb\frac{1}{2}\sum_{k=1}^{d}s_k^2(x_k-x_{k-1})\rb.
$$
We consider the Laplace transform of centered height function $h_n^{\HH}(x)$ (given in \eqref{eq:centered height functions}) with argument $\c=(c_1,\dots,c_d)\in\RR_+^d$, where we denote $\xx=(x_1,\dots,x_d)$:
\be \label{eq:definition of Laplace transform of height function}
\ph_{\xx,n}^{\HH}(\c):=\mathbb{E}_{\mu_n}\le\exp\lb-\sum_{k=1}^dc_kh_n^{\HH}(x_k)\rb\re.
\ee 

By Proposition \ref{thm:Theorem A.1 from bryc wang}, it suffices to prove that,   for all $\c$ from an open subset of $\RR_+^d$,
\be\label{eq:suffices to prove high density}
\lim_{n\rightarrow\infty}\ph_{\xx,n}^{\HH}\lb\frac{\c}{\sqrt{n}}\rb
=\exp\lb\frac{A}{2(1+A)^2}\,\sum_{k=1}^{d}\,s_k^2(x_k-x_{k-1})\rb.
\ee

We first write \eqref{eq:definition of Laplace transform of height function} explicitly,  denoting $n_k:=\floor{nx_k}$ for $k=0,\dots,d$ (note that $n_0=0$ and $n_d=n$):
\begin{equation}\label{eq:expression of phi}
    \begin{split}
        \ph_{\xx,n}^{\HH}(\c)&=\mathbb{E}_{\mu_n}\le\exp\lb-\sum_{k=1}^d\sum_{i=n_{k-1}+1}^{n_k}\lb\tau_i-\frac{A}{1+A}\rb\lb c_k+\dots+c_d\rb\rb\re\\
        &=\exp\lb\frac{A}{1+A}\,\sum_{k=1}^d\,s_k(n_k-n_{k-1})\rb
        \mathbb{E}_{\mu_n}\le\prod_{k=1}^{d}\prod_{i=n_{k-1}+1}^{n_k}\lb e^{-s_k}\rb^{\tau_i}\re\\
        &=\exp\lb\frac{A}{1+A}\,\sum_{k=1}^d\,s_k(n_k-n_{k-1})\rb
\frac{1}{Z_n}
\Pi_n\lb\underbrace{e^{-s_1},\dots,e^{-s_1}}_{n_1},\underbrace{e^{-s_2},\dots,e^{-s_2}}_{n_2-n_1},\dots,\underbrace{e^{-s_{d}},\dots,e^{-s_{d}}}_{n_{d}-n_{d-1}}\rb,
    \end{split}
\end{equation}
where $Z_n=\Pi_n(1,\dots,1)$ and in the last step we used \eqref{eq:MPA open ASEP1} of Theorem \ref{thm: stationary measure in terms of Askey--Wilson integral}.

To prove \eqref{eq:suffices to prove high density}, we first express $\ph_{\xx,n}^{\HH}(\c/\sqrt{n})$ as an integral.
Assume \eqref{eq:high density phase condition} and $A/C,ABCD\notin\{q^{-l}:l\in\NN\}$. We choose the open time interval $I$ (containing $1$) from Proposition \ref{prop:open ASEP parameters in Omega}.
We write $s_{k,n}:=s_k/\sqrt{n}$ and $t_{k,n}=e^{-s_{k,n}}$ for $k=1,\dots,d$. 
For sufficiently large $n$, we have $t_{1,n}\leq\dots\leq t_{d,n}$ in $I$.
By \eqref{PIN} of Theorem \ref{thm: stationary measure in terms of Askey--Wilson integral}, one can write:

\be \label{eq:integral expression for Pi}
\begin{split}
    \ph_{\xx,n}^{\HH}\lb\frac{\c}{\sqrt{n}}\rb&=\exp\lb\frac{A}{1+A}\,\sum_{k=1}^d\,s_{n,k}(n_k-n_{k-1})\rb
\frac{1}{Z_n}\\
&\times\int_{\RR^{d}}\prod_{k=1}^{d}\lb 1+t_{k,n}+2\sqrt{t_{k,n}}y_k \rb^{n_k-n_{k-1}}\pi_{t_{1,n},\dots,t_{d,n}}(\d y_1,\dots,\d y_{d}).
\end{split}
\ee


We have the following asymptotics of $Z_n=\Pi_n(1,\ldots,1)$:

\begin{lemma}
    Assume \eqref{eq:high density phase condition} and $A/C,ABCD\notin\{q^{-l}:l\in\NN\}$. Then 
    \be\label{eq:asymptotics of normalizing constant}
    Z_n\sim\frac{(1+A)^{2n}}{ A^n}\pp_0,
    \ee
    where we write $$\pp_0:=p_0^\aa\lb A,B,C,D\rb=\frac{(A^{-2},BC,BD,CD)_\infty}{(B/A,C/A,D/A,ABCD)_\infty}>0.$$
\end{lemma}

This result has been known before in the literature, for example in \cite{uchiyama2004asymmetric,bryc2019limit}. See Remark \ref{rmk:partition function}. 
We provide a proof for completeness.
\begin{proof} 
We first introduce some general notations. We denote $y_0(t):=y_0^{\aa}(t)=\frac{1}{2}\lb A\sqrt{t}+\frac{1}{A\sqrt{t}}\rb$. Since $A>C$ and $A>1$, after possibly shrinking the open neighborhood $I$ of $1$, we have $A\sqrt{t}>\max\lb1,\frac{C}{\sqrt{t}}\rb$ for $t\in I$. Therefore for all $t\in I$, $y_0(t)$ is the largest atom of $U_t$.
Denote
$$y_1^*(t):=\max\lb1,\frac{1}{2}\lb Aq\sqrt{t}+
\frac{1}{Aq\sqrt{t}}\rb \one_{Aq\sqrt{t}\geq1}, \frac{1}{2}\lb 
\frac{C}{\sqrt{t}}+\frac{\sqrt{t}}{C} \rb\one_{C/\sqrt{t}\geq1}\rb.$$
then for all $t\in I$ we have $y_1^*(t)<y_0(t)$, and that $\pi_t$ is supported on $U_t\subset \{ y_0(t)\}\cup[-1,y_1^*(t)]$.

    By Theorem \ref{thm: stationary measure in terms of Askey--Wilson integral}, we can write
    $Z_n=2^n\int_{\RR}(1+y)^n\pi_1(\d y)$. Note that the signed measure $\pi_1(\d y)$ is supported on $U_{1}\subset\{y_0(1)\}\cup[-1,y_1^*(1)]$, and the largest atom $y_0(1)$ has mass $\pp_0$. One can write:
    $$\frac{Z_n}{2^n}=\int_{\{y_0(1)\}}(1+y)^n\pi_1(\d y)+\int_{-1}^{y_1^*(1)}(1+y)^n\pi_1(\d y)$$
    The first term equals $(1+y_0(1))^n\pp_0=\frac{(1+A)^{2n}}{2^nA^n}\pp_0$. The second term is bounded above by $(1+y_1^*(1))^n |\pi_1| $, which converges to $0$ after divided by $(1+y_0(1))^n$, since $y_0(1)>y_1^*(1)$.
\end{proof}

Equation \eqref{eq:suffices to prove high density} (and hence the result under $A/C,ABCD\notin\{q^{-l}:l\in\NN\}$) now follows from the following:
\begin{lemma}\label{lem:4.3}
    Assume \eqref{eq:high density phase condition} and $A/C,ABCD\notin\{q^{-l}:l\in\NN\}$. 
     Define  
    $$M_n:=\int_{\RR^{d}}\prod_{k=1}^{d}\phi\lb s_{k,n},y_k\rb^{n_k-n_{k-1}}\pi_{t_{1,n},\dots,t_{d,n}}(\d y_1,\dots,\d y_{d}),$$
    where $$\phi(s,y):=\frac{1+e^{-s}+2e^{-s/2}y}{e^{-sA/(1+A)}}\frac{A}{(1+A)^2}.$$
    Assume the conditions of Lemma \ref{lem:corA}, i.e. $s_{k}\neq 2s_{k+1}$ for $1\leq k\leq d-1$. Then we have
    $$\lim_{n\rightarrow\infty}M_n = \pp_0 \exp\lb\frac{A}{2(1+A)^2}\,\sum_{k=1}^d\,s_k^2 
(x_k-x_{k-1})\rb.$$
\end{lemma}

The proof of this lemma will require the following total variation bound of Askey--Wilson signed measures:
\begin{lemma}\label{lem:corA}
    Assume $A,C\geq0$, $B,D\in(-1,0]$ and $ABCD\notin\{q^{-l}:l\in\NN\}$.
    Assume $s_1>\dots>s_d>0$ satisfying $s_k\neq 2s_{k+1}$ for $1\leq k\leq d-1$. Assume that $\theta>0$. Denote  $t_{k,n}:=e^{-s_k/\nal}$ for $1\leq k\leq d$. Then there exists $N$ and $K$ depending only on $A,B,C,D$ and $s_1,\dots,s_d$, such that for any $n\geq N$, the total variation of
    $\pi_{t_{1,n},\dots,t_{d,n}}(\d y_1,\dots,\d y_d)$
    is bounded above by $Kn^{2\al d}$. 
\end{lemma}

The above result is a direct corollary of Proposition \ref{prop:total variation} in Appendix \ref{sec:total variation bounds}.

\begin{proof}[Proof of Lemma \ref{lem:4.3}]
We first study the support of the signed measure $\pi_{t_{1,n},\dots,t_{d,n}}$.
By Lemma \ref{lem:property on support}, it is supported  in 
   $$\left\{\lb y_0\lb t_{1,n}\rb,\dots,y_0\lb t_{d,n}\rb\rb\right\}\cup\le-1,y_1^*\lb t_{1,n}\rb\re\times\le-1,y_0\lb t_{2,n}\rb\re\times\dots\times\le-1,y_0\lb t_{d,n}\rb\re.$$
Since this is the union of two disjoint sets, one can split the integral $M_n$ into two parts 
$M_n=M_{n}^1+M_{n}^2$, where 
   \begin{equation*}
       \begin{split}
           M_{n}^1&=\int_{\{\lb y_0(t_{1,n}),\dots,y_0(t_{d,n})\rb\}}\prod_{k=1}^{d}\phi\lb s_{k,n},y_k\rb^{n_k-n_{k-1}}\pi_{t_{1,n},\dots,t_{d,n}}(\d y_1,\dots,\d y_{d})\\
    &=p_0^\aa (t_{1,n})\prod_{k=1}^{d}\phi\lb s_{k,n},y_0(t_{k,n})\rb^{n_k-n_{k-1}},
       \end{split}
   \end{equation*} 
here we used the fact that the mass  of the atom $\lb y_0\lb t_{1,n}\rb,\dots,y_0\lb t_{d,n}\rb\rb\in\RR^d$ is $$p_0^\aa (t_{1,n}) =p_0^\aa(A\sqrt{t_{1,n}},B\sqrt{t_{1,n}},C/\sqrt{t_{1,n}},D/\sqrt{t_{1,n}}),$$ and
\begin{equation*}
        M_{n}^2=\int_{-1}^{y_1^*(t_{1,n})}\int_{-1}^{y_0(t_{2,n})}\dots\int_{-1}^{y_0(t_{d,n})} \prod_{k=1}^{d}\phi\lb s_{k,n},y_k\rb^{n_k-n_{k-1}}\pi_{t_{1,n},\dots,t_{d,n}}(\d y_1,\dots,\d y_{d}).
\end{equation*}
As we will see below, the dominating term will come from $M_{n}^1$.

 We first compute the asymptotics of $M_{n}^1$. As $n\rightarrow\infty$, $p_0^\aa (t_{1,n})\rightarrow\pp_0>0$. As $s\rightarrow0$,
$$\phi(s,y_0(e^{-s})) = \frac 1{1+A}e^{sA/(1+A)}+\frac A{1+A}e^{-s/(1+A)} = 1+\frac{As^2}{2(1+A)^2}+o(s^2).$$
Hence, as $n\rightarrow\infty$,
$$M_{n}^1=p_0^\aa (t_{1,n})\prod_{k=1}^{d}\phi\lb s_{k,n},y_0(t_{k,n})\rb^{n_k-n_{k-1}}\rightarrow\pp_0 \exp\lb\frac{A}{2(1+A)^2}\,\sum_{k=1}^d\,s_k^2 
(x_k-x_{k-1})\rb.$$

By Lemma \ref{lem:corA}, the total variation of $\pi_{t_{1,n},\dots,t_{d,n}}$ can be bounded by $Kn^d$, hence $M_{n}^2$ is bounded by
$$M_{n}^2 \leq Kn^d 
        \phi\lb s_{1,n},y_1^*(t_{1,n})\rb^{n_1}
        \prod_{k=2}^{d}\phi\lb s_{k,n},y_0(t_{k,n})\rb^{n_k-n_{k-1}}.$$

As $n\rightarrow\infty$, $y_1^*(t_{1,n})\rightarrow y_1^*(1)$ and $y_0(t_{1,n})\rightarrow y_0(1)$. Note that $1\leq y_1^*(1)<y_0(1)$ and $\phi(s,y)$ is strictly increasing in $y>0$, hence 
$$
\frac{M_{n}^2}{M_{n}^1}\leq\frac{Kn^d}{p_0^\aa (t_{1,n})}\lb\frac{\phi(s_{1,n},y_1^*(t_{1,n}))}{\phi(s_{1,n},y_0(t_{1,n}))}\rb^{n_1}\rightarrow0
$$
Thus the proof is concluded.
\end{proof}

\subsubsection{Stochastic sandwiching and extension}
We have already proved Theorem \ref{thm:fluctuations} in the generic sub-region $A/C,ABCD\notin\{q^{-l}:l\in\NN\}$ of high density phase \eqref{eq:high density phase condition}. We now adopt the `stochastic sandwiching' argument from \cite{corwin2021stationary} to extend equation \eqref{eq:suffices to prove high density} (and hence the result) to the whole high density phase.
We begin by recording a special case of \cite[Lemma 4.1]{corwin2021stationary}:
\begin{lemma}\label{lem:coupling lemma}
    Fix $q\in[0,1)$, and consider non-negative real numbers
    $$\alpha^{(1)}\leq\alpha^{(2)},\quad\beta^{(1)}\geq\beta^{(2)},\quad\gamma^{(1)}\geq\gamma^{(2)},\quad\delta^{(1)}\leq\delta^{(2)}.$$
    Denote by $\tau^{(j)}_i$ the occupation variable of the $i$-th site under the stationary measure $\mu_n^{(j)}$ of the $n$-site open ASEP with rates $(q,\alpha^{(j)},\beta^{(j)},\gamma^{(j)},\delta^{(j)})$, for $i\in\{1,\ldots,n\}$ and $j\in\{1,2\}$. Then there exists a coupling $\mu_n^{(1,2)}$ of the  two stationary measures $\mu_n^{(1)}$ and $\mu_n^{(2)}$  such that $\mu_n^{(1,2)}$ almost surely 
    $\tau^{(1)}_i\leq\tau^{(2)}_i$ for any $i\in\{1,\ldots,n\}$.
\end{lemma}

Fix $q\in[0,1)$. Suppose the set of parameters $(\alpha,\beta,\gamma,\delta)$ is in the high density phase (i.e. the corresponding $A,B,C,D$ satisfy $A>C$ and $A>1$). One can choose 
a sequence $(\alpha^{(j)},\beta^{(j)},\gamma^{(j)},\delta^{(j)})$, $j\in\mathbb{N}_+$ satisfying:
\begin{itemize}
    \item [(i)] $\lim_{j\rightarrow\infty}(\alpha^{(j)},\beta^{(j)},\gamma^{(j)},\delta^{(j)})=(\alpha,\beta,\gamma,\delta)$.
    \item [(ii)] $\alpha\leq\alpha^{(j)}$, $\beta\geq\beta^{(j)}$, $\gamma\leq\gamma^{(j)}$, $\delta\geq\delta^{(j)}$ for $j\in\mathbb{N}_+$.
    \item [(iii)] For $j\in \mathbb{N}_+$ the corresponding $(A^{(j)},B^{(j)},C^{(j)},D^{(j)})$ (under map \eqref{eq:defining ABCD}) belong to the high density phase \eqref{eq:high density phase condition} and satisfy  $A^{(j)}/C^{(j)},A^{(j)}B^{(j)}C^{(j)}D^{(j)}\notin\{q^{-l}:l\in\NN\}$. 
\end{itemize}
Such a sequence exists since \eqref{eq:defining ABCD} is a bijection between \eqref{eq:conditions open ASEP} and \eqref{eq:conditions qABCD}, and that 
$A^{(j)}B^{(j)}C^{(j)}D^{(j)}=\gamma^{(j)}\delta^{(j)}/\alpha^{(j)}\beta^{(j)}$.
Consider the $n$-site open ASEP with parameters  $(q,\alpha^{(j)},\beta^{(j)},\gamma^{(j)},\delta^{(j)})$, with occupation variables $\tau_i^{(j)}$ for $1\leq i\leq n$, and height function $h_n^{\HH,(j)}(x)$  for $x\in[0,1]$ (defined analogously by \eqref{eq:centered height functions}).
By Lemma \ref{lem:coupling lemma}, there exists a coupling satisfying $\tau_i\leq\tau_i^{(j)}$
for $1\leq i\leq n$, hence $h_n^{\HH}(x)\leq h_n^{\HH,(j)}(x)$ for $x\in[0,1]$. Therefore, for fixed $\xx\in\RR^d$ and $\c\in\RR^d_+$, one has
$\ph_{\xx,n}^{\HH}(\c)\geq \ph_{\xx,n}^{\HH,(j)}(\c)$.
Using \eqref{eq:suffices to prove high density} for the open ASEP with $(q,\alpha^{(j)},\beta^{(j)},\gamma^{(j)},\delta^{(j)})$, we get
$$\lim_{n\rightarrow\infty}\ph_{\xx,n}^{\HH}\lb\frac{\c}{\sqrt{n}}\rb
\geq
\lim_{n\rightarrow\infty} \ph_{\xx,n}^{\HH,(j)}\lb\frac{\c}{\sqrt{n}}\rb
=\exp\lb\frac{A^{(j)}}{2(1+A^{(j)})^2}\,\sum_{k=1}^{d}\,s_k^2(x_k-x_{k-1})\rb,$$
for  $\cc$ from the open subset of $\RR_+^d$ specified by Lemma \ref{lem:corA}, i.e. $s_{k}\neq 2s_{k+1}$ for $1\leq k\leq d-1$, where we recall that $s_k = c_k+\cdots+c_d$ for $k=1,\dots,d$.
Taking $j\rightarrow\infty$, we get \eqref{eq:suffices to prove high density} for the open ASEP with $(q,\alpha,\beta,\gamma,\delta)$ with $\geq$ instead of $=$. Using exactly the same arguments one can get the reversed inequality. We conclude that \eqref{eq:suffices to prove high density} holds  for $\mathbf{c}$ from the same open subset of $\mathbb R_+^d$. Thus the proof is concluded.
 
\subsection{Proof on the low density phase}
The result for the low density phase follows immediately from result for the
high density phase, by exactly the same argument as \cite[Section 3.2]{bryc2019limit}
using the particle-hole duality, which we also explain here for completeness. 
Consider the $n$-site open ASEP with parameters $(q,\alpha,\beta,\gamma,\delta)$. 
Instead of thinking of particles jumping around, one can view the particles as background and consider  the holes as jumping around. 
In this way, equivalently a hole jumps to the left and right sites with rates $1$ and $q$, and removed at site $1$ with rate $\alpha$ and at site $n$ with rate $\delta$, and enters site $n$ with rate $\beta$ and site $1$ with rate $\gamma$, any move is prohibited if the target site is already occupied (i.e. is a hole). 
This is exactly the $n$-site open ASEP with parameters $(q,\beta,\alpha,\delta,\gamma)$  
if we relabel the sites $\{1,\dots,n\}$ by $\{n, \dots, 1\}$.

Let $\mu_n^{(A,B,C,D)}$ denote the stationary measure of $n$-site open ASEP with parameters $(q,\alpha,\beta,\gamma,\delta)$. Let $\tau_1,\dots,\tau_n$ be occupation variables, and set $\ep_i:=1-\tau_{n-i+1}$. Denote
$$\widehat{h}_n^{\LL}(x):=\sum_{i=1}^{\floor{nx}}\lb\ep_i-\frac{C}{1+C}\rb.$$
The above particle-hole duality shows that $\{\widehat{h}_n^{\LL}(x)\}_{x\in[0,1]}$ under $\mu_n^{(A,B,C,D)}$ has the same law as $\{h_n^{\HH}(x)\}_{x\in[0,1]}$ under $\mu_n^{(C,D,A,B)}$. Therefore, Theorem \ref{thm:fluctuations} in the high density phase shows
$$\frac{1}{\sqrt{n}} \{ \widehat{h}^{\LL}_n(x)\}_{x\in[0,1]}\stackrel{\mathrm{f.d.d.}}{\Longrightarrow}   \frac{\sqrt{C}}{1+C}\{\B(x)\}_{x\in[0,1]},$$
from which one can easily show, see page 1280 in \cite{bryc2019limit}, that
$$\frac{1}{\sqrt{n}} \{ h^{\LL}_n(x)\}_{x\in[0,1]}\stackrel{\mathrm{f.d.d.}}{\Longrightarrow}   \frac{\sqrt{C}}{1+C}\{\B(x)\}_{x\in[0,1]},$$
which concludes the proof.

\section{Density profile on the coexistence line: Proof of Theorem \ref{thm:density profile coexistence line}}
\label{sec:coexistence line}
In this section we prove Theorem \ref{thm:density profile coexistence line}. We will always assume:
\be\label{eq:coexistence line condition}A,C\geq0,\quad B,D\in(-1,0],\quad A=C>1,\quad ABCD\notin\{q^{-l}:l\in\NN\}.\ee

We first compute the Laplace transform of the process $\co(x)$ given by \eqref{eq:coexistence line process}:
$$\co(x):=\frac{Ax+(1-A) \lb x\wedge U\rb }{1+A},$$
where $U\sim U(0,1)$.
For $x_0=0<x_1<\dots<x_d=1$ and  variables
 $c_1,\dots,c_d>0$, we denote $s_k = c_k+\cdots+c_d$ for $k=1,\dots,d$. For convenience we also denote $s_{d+1}=0$. We have:
\begin{equation*}
    \begin{split}
        \mathbb{E}\le\exp\lb-\sum_{k=1}^d{c_k}\co(x_k)\rb\re&=
        \sum_{l=1}^d\mathbb{E}\le\exp\lb-\sum_{k=1}^d{c_k}\co(x_k)\rb\one_{U\in[x_{l-1},x_l]}\re\\
        &=\sum_{l=1}^d\mathbb{E}\le\exp\lb-\sum_{k=1}^{l-1}c_k\frac{x_k}{1+A}-\sum_{k=l}^dc_k \frac{Ax_k+(1-A)U}{1+A}  \rb\one_{U\in[x_{l-1},x_l]}\re\\
         & =\sum_{l=1}^d\,e^{-\frac{1}{A+1}\lb\sum_{k=1}^{l-1}\,c_kx_k+A\sum_{k=l}^d\,c_kx_k\rb}\,\mathbb E\,\le e^{\frac{A-1}{A+1}s_lU} \one_{U\in[x_{l-1},x_l]}\re \\ 
        &=\frac{A+1}{A-1}\sum_{l=1}^d\,\frac{1}{s_l}\,e^{-\frac{1}{A+1}\,\left(\sum_{k=1}^{l-1}\,c_kx_k+A\,\sum_{k=l}^d\,c_kx_k\right)}\,\left(e^{\frac{A-1}{A+1}\,s_l x_l}-e^{\frac{A-1}{A+1}\,s_l x_{l-1}}\right).
    \end{split}
\end{equation*}

We consider the Laplace transform of height function $h_n(x)$ with argument $\c\in\RR_+^d$, denoting $n_k:=\floor{nx_k}$ for $k=0,\dots,d$ (note that $n_0=0$ and $n_d=n$): 
\be \label{eq:definition of Laplace transform of height function coexistence}
\begin{split}
    \Psi_{\xx,n}^{\CL}(\c):&=\mathbb{E}_{\mu_n}\le\exp\lb-\sum_{k=1}^dc_kh_n(x_k)\rb\re 
     = \mathbb{E}_{\mu_n}\le\prod_{k=1}^{d}\prod_{i=n_{k-1}+1}^{n_k}\lb e^{-s_k}\rb^{\tau_i}\re\\
    &=\frac{1}{Z_n}
\Pi_n\lb\underbrace{e^{-s_1},\dots,e^{-s_1}}_{n_1},\underbrace{e^{-s_2},\dots,e^{-s_2}}_{n_2-n_1},\dots,\underbrace{e^{-s_{d}},\dots,e^{-s_{d}}}_{n_{d}-n_{d-1}}\rb,
\end{split}
\ee 
 where  in the last step we used \eqref{eq:MPA open ASEP1} of Theorem \ref{thm: stationary measure in terms of Askey--Wilson integral} with $Z_n=\Pi_n(1,\dots,1)$. 

By Proposition \ref{thm:Theorem A.1 from bryc wang}, it suffices to prove that  for all $\c$ from an open subset of $\RR_+^d$,
\be\label{eq:suffices to prove coexistence}
\lim_{n\rightarrow\infty}\Psi_{\xx,n}^{\CL}\lb\frac{\c}{n}\rb
=\frac{A+1}{A-1}\sum_{l=1}^d\,\frac{1}{s_l}\,e^{-\frac{1}{A+1}\,\left(\sum_{k=1}^{l-1}\,c_kx_k+A\,\sum_{k=l}^d\,c_kx_k\right)}\,\left(e^{\frac{A-1}{A+1}\,s_l x_l}-e^{\frac{A-1}{A+1}\,s_l x_{l-1}}\right).
\ee

To prove \eqref{eq:suffices to prove coexistence}, we first express
$\Psi_{\xx,n}^{\CL}(\c/n)$ as an integral.
We choose the open time interval $I=(1-\ep,1)$ for some small $\ep>0$, from Proposition \ref{prop:open ASEP parameters in Omega}.
{\em Different} from the notations in Section  \ref{sec:limit fluctuation}, we now write $s_{k,n}:=s_k/n$ and $t_{k,n}=e^{-s_{k,n}}$ for $k=1,\dots,d$.
For sufficiently large $n$, we have $t_{1,n}\leq\dots\leq t_{d,n}$ in $I$. 
By \eqref{PIN} of  Theorem \ref{thm: stationary measure in terms of Askey--Wilson integral}, one can write:
\be \label{eq:integral expression for Pi coexistence}
\Psi_{\xx,n}^{\CL}\lb\frac{\c}{n}\rb=\frac{1}{Z_n}\int_{\RR^{d}}\prod_{k=1}^{d}\lb 1+t_{k,n}+2\sqrt{t_{k,n}}y_k \rb^{n_k-n_{k-1}}\pi_{t_{1,n},\dots,t_{d,n}}(\d y_1,\dots,\d y_{d}).
\ee

We now have the following asymptotics of $Z_n=\Pi_n(1,\dots,1)$:

\begin{lemma}\label{lem:asymptotics of Zn coexistence}
    Assume \eqref{eq:coexistence line condition}. Then 
    \be\label{eq:asymptotics of normalizing constant coexistence}
    Z_n\sim\fc_0\frac{A-1}{A+1}n\frac{(1+A)^{2n}}{A^n},
    \ee
    where we write $$\fc_0:= \frac{(A^{-2},AB,BD,AD)_\infty}{(B/A,q,D/A,A^2BD)_\infty}.$$
\end{lemma}

The proof of this lemma is a little technical and is deferred to the end of this section.
\begin{remark}\label{rmk:partition function}
    The quantity $\frac{Z_n}{(1-q)^n}=\ll W|(\E+\D)^n|V\rr$ is called the partition function in the physics literature. 
    The asymptotics of the partition function have been known on many parts of the phase diagram, see \cite[Remark 4.6]{bryc2019limit} for a survey. 
    Early results for some special parameters include \cite[(52),(53) and (55)]{derrida1993exact}  and \cite[(56)]{blythe2000exact}. For general parameters, \cite[(6.6) and (6.9)]{uchiyama2004asymmetric} obtained the asymptotics on $A>1$, $A>C$ and $A,C<1$.
    In the mathematical work \cite[Lemma 3.1 and Lemma 4.5]{bryc2019limit}, the asymptotics of partition function are obtained everywhere in the fan region $AC<1$. We do not find general result on the asymptotics of $Z_n$ on the coexistence line $A=C>1$ (i.e. Lemma \ref{lem:asymptotics of Zn coexistence}) in the literature.
\end{remark}

Equation \eqref{eq:suffices to prove coexistence} (and hence the result) now follows from the following:
\begin{lemma}\label{lem:5.3}
    Assume \eqref{eq:coexistence line condition}. 
    Define:
    $$H_n:=\frac{1}{n}\int_{\RR^{d}}\prod_{r=1}^{d}\psi\lb s_{r,n},y_r\rb^{n_r-n_{r-1}}\pi_{t_{1,n},\dots,t_{d,n}}(\d y_1,\dots,\d y_{d}),$$
    where $$\psi(s,y):=\lb 1+e^{-s}+2e^{-s/2}y\rb\frac{A}{(1+A)^2}.$$
    We recall that $s_k = c_k+\cdots+c_d$, $s_{k,n} =s_k/n$ and $t_{k,n}=e^{-s_{k,n}}$ for $k=1,\dots,d$.
    Assume the conditions of Lemma \ref{lem:corA}, i.e. $s_{k}\neq 2s_{k+1}$ for $1\leq k\leq d-1$.    Then we have
    \be\label{eq:limit of Hn suffices to prove}\lim_{n\rightarrow\infty}H_n = \fc_0 \sum_{l=1}^d\,\frac{1}{s_l}\,e^{-\frac{1}{A+1}\,\left(\sum_{k=1}^{l-1}\,c_kx_k+A\,\sum_{k=l}^d\,c_kx_k\right)}\,\left(e^{\frac{A-1}{A+1}\,s_l x_l}-e^{\frac{A-1}{A+1}\,s_l x_{l-1}}\right).\ee
\end{lemma}
\begin{proof}
We denote $y_0^\cc(t):=\frac{1}{2}\lb \frac{A}{\sqrt{t}}+\frac{\sqrt{t}}{A}\rb$ and $y_0^\aa(t):=\frac{1}{2}\lb A\sqrt{t}+\frac{1}{A\sqrt{t}}\rb$. Since $A>1$, after possibly making $\ep>0$ smaller (hence shrinking the interval $I=(1-\ep,1)$), we have $\frac{A}{\sqrt{t}}>A\sqrt{t}>1$ for  $t\in I$. Therefore for all $t\in I$, $y_0^\cc(t)$ and $y_0^\aa(t)$ are the largest and second largest atoms of $U_t$. 
Denote
  $$y_1^*(t):=\max\lb1,\frac{1}{2}\lb Aq\sqrt{t}+
\frac{1}{Aq\sqrt{t}}\rb \one_{Aq\sqrt{t}\geq1}, \frac{1}{2}\lb 
\frac{Aq}{\sqrt{t}}+\frac{\sqrt{t}}{Aq} \rb\one_{Aq/\sqrt{t}\geq1}\rb.$$
For all $t\in I$, we have $y_1^*(t)<y_0^\aa(t)<y_0^\cc(t)$, and that $\pi_t$ is supported on $U_t\subset \{y_0^\aa(t),y_0^\cc(t)\}\cup[-1,y_1^*(t)]$.   

We study the support of the multi-dimensional Askey--Wilson signed measure $\pi_{t_{1,n},\dots,t_{d,n}}$. By Lemma \ref{lem:property on support}, it is supported in   
$\left\{B^{0,n},\dots,B^{d,n}\right\}\cup V^{1,n}\cup\dots\cup V^{d,n}$, where for $0\leq l\leq d$, $B^{l,n}$ are points in $\RR^d$:
$$B^{l,n}:=\lb\underbrace{y_0^\cc\lb t_{1,n}\rb,\dots,y_0^\cc\lb t_{l,n}\rb}_{l},y_0^\aa\lb t_{l+1,n}\rb,\dots,y_0^\aa\lb t_{d,n}\rb \rb, $$
and for $1\leq l\leq d$, $V^{l,n}$ are compact subsets of $\RR^d$:
$$
V^{l,n}:=\prod_{i=1}^{\substack{l-1\\\longrightarrow}}\left\{y_0^\aa\lb t_{i,n}\rb,y_0^\cc\lb t_{i,n}\rb\right\}\times\le-1,y_1^*\lb t_{l,n}\rb\re\times
\prod_{i=l+1}^{\substack{d\\\longrightarrow}}\le-1,y_0^\aa\lb t_{i,n}\rb\re.
$$ 
Since the support of $\pi_{t_{1,n},\dots,t_{d,n}}$ is the union of $2d+1$ disjoint sets $\left\{B^{0,n}\},\dots,\{B^{d,n}\right\}$ and $V^{1,n},\ldots,V^{d,n}$, we will split the integral $H_n$ into $2d+1$ parts 
\be \label{eq:split of Hn}
H_n=\sum_{l=0}^dH_{l,n}^1+\sum_{l=1}^dH_{l,n}^2.
\ee 
Specifically, we define
 for $0\leq l\leq d$, 
\be \label{eq:Hln first integral}
H_{l,n}^1:=\frac{1}{n}\int_{\{B_{l,n}\}}\prod_{r=1}^{d}\psi\lb s_{r,n},y_r\rb^{n_r-n_{r-1}}\pi_{t_{1,n},\dots,t_{d,n}}(\d y_1,\dots,\d y_{d}),
\ee 
and for $1\leq l\leq d$, 
\be  \label{eq:Hln second integral}
H_{l,n}^2:=\frac{1}{n}\int_{V^{l,n}}\prod_{r=1}^{d}\psi\lb s_{r,n},y_r\rb^{n_r-n_{r-1}}\pi_{t_{1,n},\dots,t_{d,n}}(\d y_1,\dots,\d y_{d}).
\ee  
As we will see below, the dominating term will come from $\sum_{l=0}^dH_{l,n}^1$. 
\paragraph{\textbf{\underline{Step 1}}} 
We first study the limit of $H_{l,n}^1$ as $n\rightarrow\infty$, for $0\leq l\leq d$. In view of \eqref{eq:Hln first integral}, one can write
\be  \label{eq:H0n first}
H_{0,n}^1=\frac{1}{n}p_0^\aa\lb t_{1,n} \rb\prod_{r=1}^{d}\psi\lb s_{r,n},y_0^\aa(t_{r,n})\rb^{n_r-n_{r-1}},
\ee  
and for $l\geq1$,
\be \label{eq:Hln first}
\begin{split}
    H_{l,n}^1=&\frac{1}{n}p_0^\cc\lb t_{1,n}\rb\prod_{r=2}^lP^{\cc,\cc}(t_{r-1,n},t_{r,n})
    P^{\cc,\aa}(t_{l,n},t_{l+1,n})\\
    &\times\prod_{r=1}^{l}\psi\lb s_{r,n},y_0^\cc(t_{r,n})\rb^{n_r-n_{r-1}}
    \prod_{r=l+1}^{d}\psi\lb s_{r,n},y_0^\aa(t_{r,n})\rb^{n_r-n_{r-1}},
\end{split}
\ee  
where we use the shorthand 
$p_0^\aa(t):=p\lb y_0^\aa(t)\rb$, $p_0^\cc(t):=p\lb y_0^\cc(t)\rb$ and
$$P^{\cc,\cc}(t,t'):=P_{t,t'}\lb y_0^\cc\lb t\rb,\{y_0^\cc\lb t'\rb\}\rb,\quad
P^{\cc,\aa}(t,t'):=P_{t,t'}\lb y_0^\cc\lb t\rb,\{y_0^\aa\lb t'\rb\}\rb,
$$
for any $t<t'$ in $I$, where we recall that $P_{t,t'}(x,\d y)$ is defined in \eqref{eq:transition AW signed}.

For any $s>s'$, we denote $s_n:=s/n$, $s_n':=s'/n$, $t_n:=e^{-s_n}$ and $t'_n:=e^{-s'_n}$. 
As $n\rightarrow\infty$, we have: 
\begin{align*}
    p_0^\aa\lb t_{n}\rb&=\frac{\lb1/(A^2t_n),AB,BD,AD/t_n\rb_{\infty}}{\lb B/A,1/t_n,D/(At_n),A^2BD\rb_{\infty}}\sim-\fc_0\frac{n}{s},\\
    p_0^\cc\lb t_{n}\rb&=\frac{\lb t_n/A^2,ABt_n,BD,AD\rb_{\infty}}{\lb Bt_n/A,t_n,D/A,A^2BD\rb_{\infty}}\sim\fc_0\frac{n}{s},\\
    P^{\cc,\cc}(t_n,t_n')&=\frac{\lb t_n'/A^2,ABt_n',Bt_n/A,t_n \rb_{\infty}}{\lb Bt_n'/A,t_n',t_n/A^2,ABt_n\rb_{\infty}}\sim\frac{(t_n)_\infty}{(t_n')_\infty}\sim\frac{s}{s'}\\
    P^{\cc,\aa}(t_n,t_n')&=\frac{\lb 1/(A^2t_n'),AB,Bt_n/A,t_n/t_n' \rb_\infty}{\lb B/A,1/t_n',t_n/(A^2t_n'),ABt_n \rb_\infty}\sim\frac{(t_n/t_n')_\infty}{(1/t_n')_\infty}\sim\frac{s'-s}{s'}.
\end{align*}
We use the Taylor expansion of $e^{-x}$ for $x=s/n$:
\begin{align*}
    \psi\lb s_n,y_0^\aa(t_n)\rb&=\frac{1+At_n}{1+A}=1-\frac{1}{n}\frac{sA}{1+A}+O\lb\frac{1}{n^2}\rb, \\
    \psi\lb s_n,y_0^\cc(t_n)\rb&=\frac{t_n+A}{1+A}=1-\frac{1}{n}\frac{s}{1+A}+O\lb\frac{1}{n^2}\rb.
\end{align*}  
Putting them into \eqref{eq:H0n first} and \eqref{eq:Hln first}, we get, as $n\rightarrow\infty$,
\begin{align*}
    H_{0,n}^1&\sim\fc_0\lb-\frac{1}{s_1}\rb\prod_{r=1}^d\exp\lb{-\frac{s_r A\Delta x_r}{1+A}}\rb,\\
    H_{l,n}^1&\sim\fc_0\lb\frac{1}{s_l}-\frac{1}{s_{l+1}}\rb\prod_{r=1}^l \exp\lb{-\frac{s_r\Delta x_r}{1+A}}\rb
\prod_{r=l+1}^d \exp\lb{-\frac{s_r A\Delta x_r}{1+A}}\rb,\quad 1\leq l\leq d-1,\\
H_{d,n}^{1}&\sim\fc_0\frac{1}{s_d}\prod_{r=1}^d \exp\lb{-\frac{s_r\Delta x_r}{1+A}}\rb,
\end{align*} 
where we denote $\Delta x_r=x_r-x_{r-1}$ for $1\leq r\leq d$.
Summing them up, by telescoping,
\be \label{eq:sum of H1}
\begin{split}
\lim_{n\rightarrow\infty}\sum_{l=0}^dH_{l,n}^1&=\fc_0 \sum_{l=1}^d\frac{1}{s_l} \exp\lb{ -\sum_{k=1}^{l-1}\frac{s_k\Delta x_k}{1+A}-\sum_{k=l+1}^d\frac{s_k A\Delta x_k}{1+A} }\rb\lb \exp\lb{-\frac{s_l\Delta x_l}{1+A}}\rb-\exp\lb{-\frac{s_l A\Delta x_l}{1+A}}\rb\rb\\
&=\fc_0 \sum_{l=1}^d\,\frac{1}{s_l}\,\exp\lb{-\frac{1}{A+1}\,\left(\sum_{k=1}^{l-1}\,c_kx_k+A\,\sum_{k=l}^d\,c_kx_k\right)}\rb\,\left(\exp\lb{\frac{A-1}{A+1}\,s_l x_l}\rb-\exp\lb{\frac{A-1}{A+1}\,s_l x_{l-1}}\rb\right).
\end{split}
\ee 
\paragraph{\textbf{\underline{Step 2}}} 
We next study the limit of $H_{l,n}^2$ as $n\rightarrow\infty$, for $1\leq l\leq d$.
Using Lemma \ref{lem:corA} for $\theta=1$, we see that the total variation of signed measure
$\pi_{t_{1,n},\dots,t_{d,n}}$ is bounded above by $Kn^{2d}$. Therefore, in view of the fact that $\psi\lb s,y \rb$ is increasing in $y$, 
\begin{equation}\label{eq:bound of Hln2}
    \begin{split}
        H_{l,n}^2&=\frac{1}{n}\int_{V^{l,n}}\prod_{r=1}^{d}\psi\lb s_{r,n},y_r\rb^{n_r-n_{r-1}}\pi_{t_{1,n},\dots,t_{d,n}}(\d y_1,\dots,\d y_{d})\\
        &\leq \frac{1}{n}(Kn^{2d})\prod_{r=1}^{l-1}\psi\lb s_{r,n},y_0^\cc(t_{r,n})\rb^{n_r-n_{r-1}}
        \psi\lb s_{l,n},y_1^*(t_{l,n})\rb^{n_l-n_{l-1}}
        \prod_{r=l+1}^{d}\psi\lb s_{r,n},y_0^\cc(t_{r,n})\rb^{n_r-n_{r-1}}\\
        &\leq  Kn^{2d-1}\psi\lb s_{l,n},y_1^*(t_{l,n})\rb^{n_l-n_{l-1}}
    \end{split}
\end{equation}
where we have used $\psi\lb s_{r,n},y_0^\cc(t_{r,n})\rb=\frac{t_{r,n}+A}{1+A}<1$ for all $1\leq r\leq d$. Observe that,
$$\lim_{n\rightarrow\infty}\psi\lb s_{l,n},y_1^*(t_{l,n})\rb
=\frac{A}{(1+A)^2}\max\lb4,\frac{(1+Aq)^2}{Aq}\one_{Aq\geq 1}\rb<1.$$
Note that the difference $n_{l}-n_{l-1}$ is of order $n$.
In view of \eqref{eq:bound of Hln2}, we have $\lim_{n\rightarrow\infty}H_{l,n}^2=0$ for $1\leq l\leq d$.

\paragraph{\textbf{\underline{Summary of the proof}}} 
In view of \eqref{eq:split of Hn}, by \eqref{eq:sum of H1} and $\lim_{n\rightarrow\infty}H_{l,n}^2=0$ for $1\leq l\leq d$, we conclude the proof of equation \eqref{eq:limit of Hn suffices to prove}.
This finishes the proof of Lemma \ref{lem:5.3} and hence also the proof of Theorem \ref{thm:density profile coexistence line}.
\end{proof}

Finally, we give the deferred proof of  Lemma \ref{lem:asymptotics of Zn coexistence} giving   asymptotics of  partition function $Z_n=\Pi_n(1,\dots,1)$.
\begin{proof}[Proof of Lemma \ref{lem:asymptotics of Zn coexistence}]
Note that the Askey--Wilson signed measure $\pi_t(\d x)$ is defined  for  $t\in I=(1-\ep,1)$. We define
$Z_n(t):=\Pi_n(t,\dots,t)$, which is a polynomial in $t$. We have $Z_n=Z_n(1)=\lim_{t\rightarrow 1^-}Z_n(t)$, and for $t\in I$,
\be\label{eq:Zn(t)}Z_n(t)=\int_\RR\lb 1+t+2\sqrt{t}x\rb^n\pi_t(\d x).\ee
Recall that $A>1$. Assume $m$ is the largest integer such that $Aq^m\geq 1$.
If $A>q^{-m}$, then for $t\rightarrow1^-$, $U_t$ has atoms $y_j^\cc(t):=\frac{1}{2}\lb \frac{Aq^j}{\sqrt{t}}+\frac{\sqrt{t}}{Aq^j}\rb$ and $y_j^\aa(t):=\frac{1}{2}\lb Aq^j\sqrt{t}+\frac{1}{Aq^j\sqrt{t}}\rb$ for $0\leq j\leq m$. If $A=q^{-m}$, then  for $t\rightarrow1^-$, $U_t$ has atoms $y_j^\cc(t)$ for $0\leq j\leq m$ and $y_j^\aa(t)$ for $0\leq j\leq m-1$. The atom masses are
\begin{align}\label{eq:masses}
    p_j^\cc(t)&=\frac{\lb t/A^2,ABt,BD,AD \rb_\infty}{\lb Bt/A,t,D/A,A^2BD \rb_\infty}
\frac{q^j\lb1-A^2q^{2j}/t\rb\lb A^2/t,AB,A^2,AD/t\rb_j}{(q)_j(1-A^2/t)(A^2/t)^j\prod_{l=1}^j\lb(Bt/A-q^l)(t-q^l)(D/A-q^l)\rb},\\
p_j^\aa(t)&=\frac{\lb 1/(A^2t),AB,BD,AD/t \rb_\infty}{\lb B/A,1/t,D/(At),A^2BD \rb_\infty}
\frac{q^j(1-A^2q^{2j}t)\lb A^2t,ABt,A^2,AD\rb_j}{(q)_j(1-A^2t)(A^2t)^{2j}\prod_{l=1}^j\lb(B/A-q^l)(1/t-q^l)(D/(At)-q^l)\rb}.\label{eq:masses2}
\end{align} 
The equations above give finite numbers since the Askey--Wilson signed measures are well-defined. 

We now analyze the limits of the atoms as $t\rightarrow 1^-$.
For each  $j\in\NN$, if the atoms corresponding to $y_j^\cc(t)$ and $y_j^\aa(t)$ both exist, then  their positions converge to the same limit $y_j(1):=\frac{1}{2}\lb Aq^j+\frac{1}{Aq^j}\rb$. 
We next consider the limits as $t\rightarrow 1^-$ for their masses $\{p_j^\cc(t),p_j^\aa(t)\}$ given by \eqref{eq:masses} and \eqref{eq:masses2} above. We observe that the numerators converge to the same nonzero finite value. In the denominators, except for the term $1-t$ coming from $(t)_{\infty}$ in $p_j^\cc(t)$ and the term $1-1/t$ coming from $(1/t)_{\infty}$ in $p_j^\aa(t)$, all other terms converge to nonzero finite values, and the limits of terms in $p_j^\cc(t)$ are equal to their counterparts in $p_j^\aa(t)$.  In summary, we conclude that as $t\rightarrow 1^-$, one of the masses in $\{p_j^\cc(t),p_j^\aa(t)\}$ approaches $+\infty$ while the other approaches $-\infty$, and both of them approach $\pm\infty$ as constant times $ \frac{1}{1-t}$. 
In the special case $A=q^{-m}$, the atom corresponding to $y_m^\cc(t)$ exists, but the atom corresponding to $y_m^\aa(t)$ does not exist. By a similar observation, $p_m^\cc(t)$ approaches a finite constant as $t\rightarrow 1^-$.

One can write $Z_n(t)=Z_{n,\cont}(t)+\sum_{j=0}^mZ_{n,j}(t)$, where 
$$
Z_{n,j}(t)=\int_{\{y_j^\cc(t),y_j^\aa(t)\}}\lb 1+t+2\sqrt{t}x\rb^n\pi_t(\d x),\quad
Z_{n,\cont}(t)=\int_{-1}^1\lb 1+t+2\sqrt{t}x\rb^n\pi_t(\d x).
$$
In the special case $A=q^{-m}$ and $j=m$, the atom corresponding to $y_m^\aa(t)$ does not exist, and $Z_{n,m}(t)$ is the integral on the single atom $\{y^\cc_{m}(t)\}$. 

We will show below that as $t\rightarrow 1^-$, $Z_{n,\cont}(t)$ and $Z_{n,j}(t)$ for $0\leq j\leq m$ converge to finite numbers, which we denote by $Z_{n,\cont}$ and $Z_{n,j}$. Therefore
we have 
\be  \label{eq:decomposition of Zn}
Z_n=Z_{n,\cont}+\sum_{j=0}^mZ_{n,j}.
\ee  

\paragraph{\textbf{\underline{Step 1}}} We first study $Z_{n,0}$. For $t\in I$, one can write:
\begin{equation*}
    \begin{split}
        Z_{n,0}(t)
        =&\lb1+t+2\sqrt{t}y_0^\aa(t)\rb^np_0^\aa(t)+\lb1+t+2\sqrt{t}y_0^\cc(t)\rb^np_0^\cc(t)\\
        =&(1+A)^n\lb t+\frac{1}{A}\rb^n\frac{\lb \frac{1}{A^2t},AB,BD,\frac{AD}{t} \rb_\infty}{\lb \frac{B}{A},\frac{q}{t},\frac{D}{At},A^2BD \rb_\infty}\frac{t}{t-1}
        +(t+A)^n\frac{(1+A)^n}{A^n}\frac{\lb \frac{t}{A^2},ABt,BD,AD \rb_\infty}{\lb \frac{Bt}{A},qt,\frac{D}{A},A^2BD \rb_\infty}\frac{1}{1-t}\\
        =&\frac{(1+A)^n}{A^n}\frac{1}{t-1}\le(At+1)^nt\frac{\lb \frac{1}{A^2t},AB,BD,\frac{AD}{t} \rb_\infty}{\lb \frac{B}{A},\frac{q}{t},\frac{D}{At},A^2BD \rb_\infty}
        -(t+A)^n\frac{\lb \frac{t}{A^2},ABt,BD,AD \rb_\infty}{\lb \frac{Bt}{A},qt,\frac{D}{A},A^2BD \rb_\infty}\re.
    \end{split}
\end{equation*}
By the L'Hospital rule, we have 
\begin{equation*}
    \begin{split}
        Z_{n,0}=\lim_{t\rightarrow 1^-}Z_{n,0}(t)=&\frac{(1+A)^n}{A^n}(1+A)^n\fc_0\left[
        \frac{nA}{A+1}+1-\frac{n}{A+1}+\pone\log\lb\frac{1}{A^2t}\rb_{\infty}+\pone\log\lb\frac{AD}{t}\rb_{\infty} \right.\\
        &+\pone\log\lb\frac{1}{\lb\frac{q}{t}\rb_{\infty}}\rb+\pone\log\lb\frac{1}{\lb\frac{D}{At}\rb_{\infty}}\rb - \pone\log\lb\frac{t}{A^2}\rb_\infty- \pone\log\lb ABt\rb_\infty\\
        &\left.- \pone\log\lb\frac{1}{\lb\frac{Bt}{A}\rb_\infty}\rb- \pone\log\lb\frac{1}{\lb qt\rb_\infty}\rb\right]. 
    \end{split}
\end{equation*} 
Therefore as $n\rightarrow\infty$, $$Z_{n,0}\sim\fc_0\frac{A-1}{A+1}n\frac{(1+A)^{2n}}{A^n},$$
where we recall that $$\fc_0:= \frac{(A^{-2},AB,BD,AD)_\infty}{(B/A,q,D/A,A^2BD)_\infty}.$$

\paragraph{\textbf{\underline{Step 2}}} Next, for $j\geq 1$, we study  $Z_{n,j}$. Recall that $y_j(1):=y_j^\cc(1)=y_j^\aa(1)=\frac{1}{2}\lb Aq^j+\frac{1}{Aq^j}\rb$.

We first take care of the special case $A=q^{-m}$ and $j=m$. We have
$Z_{n,m}(t)=\lb1+t+2\sqrt{t}y_m^\cc(t)\rb^np_m^\cc(t).$
As $t\rightarrow1^-$, $p_m^\cc(t)$ approaches a finite constant, hence
$$Z_{n,j}=\lim_{t\rightarrow1^-}Z_{n,m}(t)=\lb2+2y_m(1)\rb^n\lim_{t\rightarrow1^-}p_m^\cc(t),$$
which, after divided by 
$\frac{(1+A)^{2n}}{A^n}n=(2+2y_0(1))^nn$, converges to $0$ as $n\rightarrow\infty$.

In other cases, the atoms corresponding to $y_j^\cc(t)$ and $y_j^\aa(t)$ both exist, and 
$$Z_{n,j}(t)=\lb1+t+2\sqrt{t}y_j^\aa(t)\rb^np_j^\aa(t)+\lb1+t+2\sqrt{t}y_j^\cc(t)\rb^np_j^\cc(t).$$
One can observe that, both $p_j^\aa(t)(1-t)$ and $p_j^\cc(t)(1-t)$ can be extended to analytic functions for $t$ in a small neighborhood around $1$. 
By L'Hospital rule, one can take the limit:
$$Z_{n,j}=\lim_{t\rightarrow1^-}Z_{n,j}(t)=-\pone\lb
\lb1+t+2\sqrt{t}y_j^\aa(t)\rb^np_j^\aa(t)(1-t)+\lb1+t+2\sqrt{t}y_j^\cc(t)\rb^np_j^\cc(t)(1-t)
\rb.$$
Therefore, as $n\rightarrow\infty$, 
$$Z_{n,j}\sim\text{Const}\times(2+2y_j(1))^nn,$$ 
which, after divided by 
$\frac{(1+A)^{2n}}{A^n}n=(2+2y_0(1))^nn$, converges to $0$ since $y_0(1)>y_j(1)$ for $j\geq 1$ and such that $Aq^j\ge 1$.

\paragraph{\textbf{\underline{Step 3}}} 
Last, we study $Z_{n,\cont}$.  
The continuous density \eqref{eq:continuous part density} equals
$$f(t;x):=
\frac{(q, ABt, A^2, AD, AB, BD, AD/t)_{\infty}}{2\pi
(A^2BD)_{\infty}\sqrt{1-x^2}} \biggl|\frac{(e^{2\i\theta
})_{\infty}}{(A\sqrt{t}e^{\i\theta}, B\sqrt{t}e^{\i\theta}, Ae^{\i\theta}/\sqrt{t},
De^{\i\theta}/\sqrt{t})_{\infty}} \biggr|^2\one_{|x|<1}.$$ 
If $A=q^{-r}$ for some $r\in\NN$, the continuous density always equals $0$ due to the term $(A^2)_{\infty}$ in the numerator, hence $Z_{n,\cont}(t)=0$. In other cases, for $t$ in a small interval around $1$, since the linear terms in the denominator are uniformly bounded away from $0$, one can bound $|f(t,x)|$ above by a finite uniform constant. Hence 
$Z_{n,\cont}=\lim_{t\rightarrow1^-}Z_{n,\cont}(t)\leq\text{Const}\times4^n,$
which, after divided by 
$\frac{(1+A)^{2n}}{A^n}n>4^nn$, converges to $0$.
\paragraph{\textbf{\underline{Summary of the proof}}} 
In Steps 1-3, we have proved that, as $n\rightarrow\infty$,
$Z_{n,0}\sim\frac{(1+A)^{2n}}{A^n}n\frac{A-1}{A+1}\fc_0$,
and both $Z_{n,j}$ for $j\geq 1$ and $Z_{n,\cont}$ converges to $0$ as $n\rightarrow\infty$, after divided by 
$\frac{(1+A)^{2n}}{A^n}n.$
Hence, by \eqref{eq:decomposition of Zn}, we have 
$Z_n\sim\frac{(1+A)^{2n}}{A^n}n\frac{A-1}{A+1}\fc_0$.
\end{proof}

\appendix
\section{Total variation bounds of Askey--Wilson signed measures}
\label{sec:total variation bounds}

In this appendix we prove a technical total variation bound for the Askey--Wilson signed measures. Although we state and prove the results for general $\al>0$, we will only need the case $\al=1/2$ in Section \ref{sec:limit fluctuation} and $\al=1$ in Section \ref{sec:coexistence line}.
\begin{proposition}
    \label{prop:total variation}
    Assume $A,C\geq0$, $B,D\in(-1,0]$ and $ABCD\notin\{q^{-l}:l\in\NN\}$. Assume $s>s'>0$,
    $s\neq 2s'$, and $\al>0$. 
    We denote $t_n:=e^{-s/\nal}$ and $t'_n:=e^{-s'/\nal}$ for any $n\in\mathbb{N}_+$. 
    Then there exists $N$ and $K$ depending only on $A,B,C,D,s,s'$, such that for any $n\geq N$, the total variations of $\pi_{t_n}(\d x)$ and of $P_{t_n,t_n'}(x,\d y)$ for any $x\in U_{t_n}$ are both bounded above by $Kn^{2\al}$. 
\end{proposition} 
\begin{proof}
 Note that  $t_n<t_n'<1$ and for sufficiently large $n$, points   $t_n,  t_n'$ are in the open interval $I$ from Proposition \ref{prop:open ASEP parameters in Omega}. We also have $B\sqrt{t_n},B\sqrt{t_n'},D/\sqrt{t_n},D/\sqrt{t_n'}\in(-1,0]$, in particular $U_{t_n},U_{t'_n}\subset[-1,\infty)$. In view of \eqref{eq:mass equal 1} and Lemma \ref{lem:support of signed measure Pst}, to bound total variations of $\pi_{t_n}(\d x)$ and $P_{t_n,t_n'}(x,\d y)$, one needs to bound all the masses of the atoms. In our case all atoms are $\geq 1$. 

In view of the fact that $t_n,t_n'\in I$ as $n\rightarrow\infty$, the four entries in 
\be\label{eq:formula for marginal in proof}\pi_{t_n}(\d x)=\nu\lb\d x; A\sqrt{t_n},B\sqrt{t_n},C/\sqrt{t_n},D/\sqrt{t_n}\rb,\ee
and, for $x\in U_{t_n}$, the four entries in
\be\label{eq:formula for transition in proof}P_{t_n,t_n'}(x,\d y)=\nu\lb\d y; A\sqrt{t_n'},B\sqrt{t_n'},\sqrt{\frac{t_n}{t_n'}}\lb x+\sqrt{x^2-1}\rb,\sqrt{\frac{t_n}{t_n'}}\lb x-\sqrt{x^2-1}\rb\rb,\ee
have uniformly bounded norms. In particular, the total number of atoms is uniformly bounded.

We look at the formulas for atom masses \eqref{eq: p_0} and \eqref{eq: p_j}: For $|aq^j|\geq1$,
\be\label{eq:formulas for atom masses in proof}
\begin{split}
    p_0^\aa&=\frac{(a^{-2},bc,bd,cd)_\infty}{(b/a,c/a,d/a,abcd)_\infty},\\
    p_j^\aa&=\frac{(a^{-2},bc,bd,cd)_\infty}{(b/a,c/a,d/a,abcd)_\infty}\frac{q^j(1-a^2q^{2j})(a^2,ab,ac,ad)_j}{(q)_j(1-a^2)a^{4j}\prod_{l=1}^{j}\lb(b/a-q^l)(c/a-q^l)(d/a-q^l)\rb},\quad j\geq 1.
\end{split}
\ee 
The numerators of these masses are uniformly bounded from above, and one needs to bound the denominators away from $0$. Since $abcd$ equals either $ABCD$ (in case of $\pi_{t_n}(\d x)$) or $ABt_n$ (in case of $P_{t_n,t_n'}(x,\d y)$), it is uniformly bounded away from $\{q^{-l}: l\in\NN\}$  (for $ABCD$ it follows by the assumption and for $ABt_n$ due to the fact that $B\leq 0$). Moreover, when $|aq^j|\geq 1$ for some $j\geq 1$, the term $1-a^2$ is uniformly bounded away from $0$. 
Note that we always have $b\in(-1,0]$.
Consequently, by swapping $a$ with $\{c,d\}$, we claim that it suffices to prove the following criterion:
\be\label{eq:crit}
\text{For any two distinct }\e,\f\in\{a,c,d\}\text{ satisfying }\e\geq 1,\text{ we have }|\f/\e|=re^{\ep/\nal}\text{ for some }r\geq 0, \ep\neq 0.
\ee 
If condition \eqref{eq:crit} holds, then for any $l\in\mathbb{Z}$, 
$|1-q^l\f/\e|$ is bounded from below by a uniform positive constant if $r\neq q^{-l}$, and bounded from below by a uniform positive constant times $1/\nal$ if $r=q^{-l}$. Looking at the formulas of atom masses \eqref{eq:formulas for atom masses in proof}, in the denominator of $p_j^\e$ for $j\geq 0$ and $\e\geq 1$, most of the linear terms in the $q$-Pochhammer symbols  are bounded from below by a  positive constant, except for at most two of them that are bounded from below by a  positive constant times $1/\nal$. Hence  $|p_j^\e|$ is bounded from above by a uniform constant times $n^{2\al}$.

We now turn to the proof of condition \eqref{eq:crit}. We first consider $\pi_{t_n}(\d x)$ given by \eqref{eq:formula for marginal in proof}. We have $B\sqrt{t_n},D/\sqrt{t_n}\in(-1,0]$, and $\left|\frac{A\sqrt{t_n}}{C/\sqrt{t_n}}\right|=At_n/C=(A/C)e^{-s/\nal}$, so \eqref{eq:crit} is satisfied. 

We then consider $P_{t_n,t_n'}(x,\d y)$ given by \eqref{eq:formula for transition in proof}, for $x\in U_{t_n}$. We split it into the following cases: 

\begin{itemize}
    \item [Case 1.] $x\in[-1,1]$. Then $\sqrt{\frac{t_n}{t_n'}}\lb x+\sqrt{x^2-1}\rb$ and $\sqrt{\frac{t_n}{t_n'}}\lb x-\sqrt{x^2-1}\rb$ are complex conjugate with norm $<1$. 
    Note that $\vert\sqrt{\frac{t_n}{t_n'}}\lb x\pm\sqrt{x^2-1}\rb/\lb A\sqrt{t_n'}\rb\vert=\sqrt{t_n}/(At_n')=(1/A)e^{(s'-s/2)/\nal}$. In view of the assumption $s\neq 2s'$,  \eqref{eq:crit} is satisfied.
    \item [Case 2.] $x=\frac{1}{2}\lb q^kC/\sqrt{t_n}+(q^kC/\sqrt{t_n})^{-1}\rb$ and $q^kC/\sqrt{t_n}>1$. Then $\sqrt{\frac{t_n}{t_n'}}\lb x+\sqrt{x^2-1}\rb=Cq^k/\sqrt{t_n'}$
    and $\sqrt{\frac{t_n}{t_n'}}\lb x-\sqrt{x^2-1}\rb=t_n/(q^kC\sqrt{t_n'})$. The three ratios between $A\sqrt{t_n'}$, $Cq^k/\sqrt{t_n'}$ and $t_n/(q^kC\sqrt{t_n'})$ are respectively
    $(A/(Cq^k))e^{-s'/\nal}$, $(C^2q^{2k})e^{s/\nal}$ and $(ACq^k)e^{(s-s')/\nal}$, so \eqref{eq:crit} is satisfied.
    \item [Case 3.] $x=\frac{1}{2}\lb q^kA\sqrt{t_n} +(q^kA\sqrt{t_n})^{-1}\rb$ and $q^kA\sqrt{t_n}>1$. Then $\sqrt{\frac{t_n}{t_n'}}\lb x+\sqrt{x^2-1}\rb=(A q^k)t_n/\sqrt{t_n'}$
    and $\sqrt{\frac{t_n}{t_n'}}\lb x-\sqrt{x^2-1}\rb=1/(Aq^k\sqrt{t_n'})$. The three ratios between $A\sqrt{t_n'}$, $(A q^k)t_n/\sqrt{t_n'}$ and $1/(Aq^k\sqrt{t_n'})$ are respectively $q^{-k}e^{(s-s')/\nal}$, $A^2q^{2k}e^{-s/\nal}$ and $A^2q^ke^{-s'/\nal}$ so \eqref{eq:crit} is satisfied.
\end{itemize}
Thus the proof is concluded.
\end{proof}
 
\bibliographystyle{goodbibtexstyle}
\bibliography{AskeyWilson}

\end{document}